%% file: On_the_gluing_formula_of_real_analytic_torsion_forms_.tex
\documentclass[12pt, centertags]{amsart}
\usepackage{amsmath, amstext, amsthm, a4, amssymb, amscd}
\usepackage{amssymb, graphics, epsfig, color}
\usepackage[mathscr]{eucal}
\usepackage[all]{xy}
\usepackage{mathrsfs}
\usepackage{imakeidx}
\makeindex

\numberwithin{equation}{section}

 \def\Im{{\rm Im}}

 \DeclareMathOperator{\Ker}{Ker}  
  \DeclareMathOperator{\rk}{rk} \DeclareMathOperator{\Id}{Id} 
 \DeclareMathOperator{\tr}{Tr}

 \newtheorem{thm}{Theorem}[section] \newtheorem{lemma}[thm]{Lemma} \newtheorem{conj}[thm]{Conjecture} \newtheorem{prop}[thm]{Proposition}
  \theoremstyle{definition} \newtheorem{rem}[thm]{Remark} \theoremstyle{definition}
 \newtheorem{defn}[thm]{Definition} %
 \theoremstyle{remark} %
 \newcommand{\be}{\begin{eqnarray}}
   
 \newcommand{\comment}[1]{}

\makeindex

\begin{document}

 \title{On the gluing formula of real analytic torsion forms}
 \date{\today}
 \author{Jialin Zhu}
 \address{Chern Institute of Mathematics, Nankai University, Tianjin 300071, P.R. China}
 \email{jialinzhu@nankai.edu.cn}

\begin{abstract} In this paper we extend first the Bismut-Lott's analytic torsion form for flat vector bundles to the boundary case, then we establish its gluing formula on
 a smooth fibration under the assumption that a fiberwise Morse function exists.
We assume that the metrics have product structures near the cutting hypersurface.
\end{abstract}

\maketitle
\tableofcontents
\setcounter{section}{-1}

\section{Introduction}\label{s1}
 In 1935, Reidemeister, Franz and de Rham introduced what we now call Reidemeister-Franz torsion
  (RF-torsion) for certain finite
simplicial complexes
(cf. \cite{BisZh92},  \cite{Franz35}, \cite{Milnor66}, \cite{RaySing71}, \cite{Reid35}). RF-torsion is the first algebraic-topological
 invariant which can distinguish the homeomorphism types of homotopy-equivalent lens spaces known
 at that time (cf. \cite[$\S$12]{Milnor66}).
 As its analytic analogue, Ray
 and Singer \cite{RaySing71} introduced what we now call Ray-Singer analytic torsion associated to de Rham complex twisted by a flat vector bundle $F$ over a
compact oriented Riemannian
manifold $M$. They constructed this new analytic invariant in term of the zeta function for Hodge Laplacian, and they conjectured that RF-torsion and the analytic torsion should be equivalent for unitary representations.
This conjecture was proved by Cheeger \cite{Chg} and M\"{u}ller \cite{Mu78} independently by using different methods.
In 1992 Bismut and Zhang \cite{BisZh92} and M\"{u}ller \cite{Mu93} simultaneously considered its generalizations. M\"{u}ller extended
his result to the case where the dimension of the manifold is odd and only the metric induced on $\det{F}$ is required to be flat.
Bismut and Zhang generalized the original Cheeger-M\"{u}ller theorem to arbitrary flat vector bundles with arbitrary Hermitian
metrics. There are also various extensions to the equivariant case (cf. \cite{BisZh94}, \cite{LoRo91}, \cite{Luck93}). In particular, Bismut and Zhang \cite{BisZh94} extended their results in \cite{BisZh92} to the equivariant situation.

In \cite{BL}, Bismut and Lott constructed what we now call Bismut-Lott torsion form (BL-torsion)  for smooth fibrations
with compact
fiber as generalization of Ray-Singer analytic torsion, which will be discussed in detail later.
 Inspired by the work of Bismut and Lott,
Igusa
\cite{Igusa02} can finally complete the construction of higher RF-torsion by using the parameterized Morse theory.
 The reader refers to the books of Igusa \cite{Igusa02} and \cite{Igu05} for more information about the higher Igusa-Klein
torsion (IK-torsion). A second version of higher RF-torsion (DWW-torsion) was defined by Dwyer, Weiss and Williams \cite{DwWeiWill03} in the homotopy theoretical approach.
Bismut and Goette \cite{BGo} obtained a family version of the Bismut-Zhang Theorem under the assumption that there exists a fiberwise
Morse function for the fibration in question. Goette \cite{Goette01}, \cite{Goette03} did more work towards the precise relation on BL-torsion and IK-torsion. The survey
\cite{Goette08} of Goette gives an overview about these higher torsion invariants for families. The
reader can refer to \cite{BGo}, \cite{Bunke00} for the equivariant BL-torsion form and to \cite{BisMaZh} for the recent works on the analytic torsion form.

In Igusa's axiomatization of higher torsion invariants (cf. \cite[$\S$3]{Igu08}), he summarized two axioms: Additivity Axiom and Transfer Axiom,
 to characterize the higher torsions, up to an universal cohomology class depending only on the underlaying manifold. In \cite[$\S$5]{Igu08}, Igusa established the additivity formula and the
transfer formula for IK-torsion. Roughly speaking, the additivity
formula of IK-torsion corresponds to the gluing formula of BL-torsion, and the transfer formula of IK-torsion corresponds to
the functoriality of BL-torsion with respect to the composition of two submersions, which
has been established by Ma \cite{Ma02}.
The main results of Igusa in \cite{Igu08} were first developed and announced during
the conference \cite{Conf03} on the higher torsion invariants in G\"{o}ttingen in September 2003.
To study the gluing problem of BL-torsion was proposed as an
open problem during this conference in order to clarify the relation between BL-torsion and IK-torsion.
Once we have established the gluing formula for BL-torsion, then it will imply basically that there exist a constant
$c$ and a cohomology class $R\in H^{*}(S)$ such than $\tau_{\rm IK}=c\tau_{\rm BL}+R$,
when they are well-defined as cohomology classes. This is the main motivation of the present work.

L\"{u}ck \cite{Luck93} established the gluing formula for the analytic torsion for unitary flat vector bundles when the Riemannian metric has product structure near the boundary by using the results in \cite{LoRo91}.
 There are also other works on the gluing problem of the
 analytic torsion (cf. \cite{Has},  \cite{Vishik95}). Finally,
Br\"{u}ning and Ma \cite{BruMa06} established the anomaly formula of the analytic torsion on manifolds with boundary,
then they \cite{BruMa12} proved the gluing formula of analytic torsion for any flat vector bundles and without any assumptions on the product structures near the boundary.

Now let us state some of our results in detail. First, we recall some properties of Bismut-Lott torsion form. Let M be a smooth manifold without boundary. Let $\big(F, \nabla^{F}\big)$ be a flat complex vector
 bundle on $M$ of rank $\rk(F)$ with a flat connection $\nabla^{F}$, i.e., $(\nabla^{F})^{2}=0$. Let $h^{F}$ be a Hermitian metric
 on $F$. Put
\begin{align}\begin{aligned}\label{e.356}
\omega(F, h^{F})=(h^{F})^{-1}\nabla^{F}h^{F}\in \Omega^{1}(M,{\rm End}(F)),
\end{aligned}\end{align}
which is an $\text{End}(F)$-valued 1-form on $M$. For $k\in \mathbb{N}$, let
\begin{align}\begin{aligned}\label{e.356}
c_{2k+1}(F, h^{F})=(2i\pi)^{-k}2^{-(2k+1)}\tr[\omega^{2k+1}(F, h^{F})],
\end{aligned}\end{align}
then $c_{2k+1}(F, h^{F})$ is a real closed $(2k+1)-$form on $M$ (cf. \cite[(1.34)]{BL}). Its de Rham cohomology class $c_{2k+1}(F)\in H^{*}(M, \mathbb{R})$ does
not depend
 on $h^{F}$.

 Let $S$ be a compact smooth manifold. Let $\pi: M\rightarrow S$ be a smooth fibration (cf. \cite[Chapter 1]{BGV92}), whose standard fiber
$Z$ is a compact $m-$dimensional smooth manifold. Let $H^{p}(Z, F), 0\leq p\leq m,$ denote the complex vector bundle on $S$ whose
fiber is the cohomology group $H^{p}(Z_{b}, F)$ at $b\in S$. Then $H^{p}(Z, F)$ admits the canonical flat connection $\nabla^{H^{p}(Z, F)}$
 induced by $\nabla^{F}$ (cf. \cite[$\S$3\,(f)]{BL}).

Let $TZ$ be the vertical tangent bundle of $M$, $o(TZ)$ be its orientation bundle, a flat real line bundle on $M$, and $e(TZ)\in H^{m}(M, o(TZ))$
be the Euler class of $TZ$, then Bismut and Lott \cite{BL} have proved that for $k\in \mathbb{N}$
\begin{align}\begin{aligned}\label{e.357}
\sum_{p=0}^{m}(-1)^{p}c_{2k+1}(H^{p}(Z, F))=\int_{Z}e(TZ)\cdot c_{2k+1}(F) \quad \text{ in } H^{2k+1}(S, \mathbb{R}).
\end{aligned}\end{align}
One sees that (\ref{e.357}) is an analog of the Riemann-Roch-Grothendieck theorem for holomorphic submersions.

Equip the fibration with a horizontal distribution $T^{H}M$ such that $T^{H}M\oplus TZ=TM$ and a vertical Riemannian metric $g^{TZ}$,
 then in \cite[Def. 3.22]{BL} Bismut and Lott constructed a natural real $2k-$form $\mathscr{T}_{2k}(T^{H}M, g^{TZ}, h^{F})$, for $k\in\mathbb{N}$, on $S$ such that
\begin{align}\begin{aligned}\label{e.358}
 d(\mathscr{T}_{2k}(T^{H}M, g^{TZ}, h^{F}))=&\int_{Z}e(TZ, \nabla^{TZ})\cdot c_{2k+1}(F, h^{F}) \\
 &-\sum_{p=0}^{m}(-1)^{p}c_{2k+1}(H^{p}(Z, F), h^{H^{p}(Z, F)}).
\end{aligned}\end{align}
Here $e(TZ,\nabla^{TZ})\in \Omega^{m}(M,o(TZ))$ is the Euler form associated with the canonical connection $\nabla^{TZ}$ on $TZ$. We call $\mathscr{T}_{2k}(T^{H}M, g^{TZ}, h^{F})$ the Bismut-Lott analytic torsion forms. In \cite[Thm. 3.29]{BL}, they showed that
the $0-$form $\mathscr{T}_{0}(T^{H}M, g^{TZ}, h^{F})$ at $b\in S$
is equal to half of Ray-Singer analytic torsion of the fiber $Z_{b}$ with $F|_{Z_{b}}$, so it's a natural higher degree generalization of
the Ray-Singer analytic torsion. Bismut and Lott also showed that
under some appropriate conditions $\mathscr{T}_{2k}(T^{H}M,g^{TZ},h^{F})$ is closed and its de Rham cohomology class $\mathscr{T}_{2k}(M,F)$ is independent of the choices
of $T^{H}M$, $g^{TZ}$ and $h^{F}$ (cf. \cite[Cor. 3.25]{BL}), thus $\mathscr{T}_{2k}(M, F)\in H^{2k}(S, \mathbb{R})$ is a smooth invariant of the pair $(M\overset{\pi}\rightarrow S, F)$.\\

In this paper, we will consider the gluing problem of the Bismut-Lott torsion form.
We suppose that $X$ is a compact hypersurface in $M$ such that $M=M_{1}\cup_{X} M_{2}$ and $M_{1}, M_{2}$ are
 manifolds with the common boundary $X_{1}=X_{2}=X$. We also assume that
$$
Z_{1}\rightarrow M_{1}\overset{\pi}\rightarrow S, \quad Z_{2}\rightarrow M_{2}\overset{\pi}\rightarrow S, \quad \text{and } Y\rightarrow X\overset{\pi}\rightarrow S
$$
are all smooth fibrations with fiber $Z_{1, b}$, $Z_{2, b}$ and $Y_{b}$ at $b\in S$ such that
\begin{align}\begin{aligned}\label{e.359}
Z_{b}=Z_{1, b}\cup_{Y_{b}}Z_{2, b}.
\end{aligned}\end{align}
In other words, the fibrations $M_{1}$ and $M_{2}$ can be glued into $M$ along $X$ (cf. Figure \ref{fig.5}).

\begin{figure}
\input{fibration.pstex_t}
\caption[Figure]{\label{fig.5}}
\end{figure}\index{Figure \ref{fig.5}}

 Let $U_{\varepsilon}\simeq X\times (-\varepsilon, \varepsilon)$ be a product neighborhood of $X$ in $M$, and $\psi_{\varepsilon}:X\times (-\varepsilon, \varepsilon)\rightarrow X$ be the projection on the first factor.
We \textbf{assume} that $T^{H}M$ and $g^{TZ}$ have product structures on $U_{\varepsilon}$, i.e.,
\begin{align}\begin{aligned}\label{e.759}
(T^{H}M)|_{X}\subset TX,\quad\big(T^{H}M\big)|_{U_{\varepsilon}}=\psi_{\varepsilon}^{*}\big((T^{H}M)|_{X}\big),
\end{aligned}\end{align}
\begin{align}\begin{aligned}\label{e.360}
g^{TZ}|_{(x', x_{m})}=dx^{2}_{m}+g^{TY}(x'), \quad (x', x_{m})\in X\times (-\varepsilon, \varepsilon).
\end{aligned}\end{align}
Then $T^{H}X:=(T^{H}M)|_{X}$ gives a horizontal bundle of fibration $X$, such that $TX=T^{H}X\oplus TY$. We trivialize $F$ along $x_{m}$-direction, by using the parallel transport with respect to the flat connection $\nabla^{F}$, then we have
\begin{align}\begin{aligned}\label{e.760}
\quad (F, \nabla^{F})|_{X \times (-\varepsilon, \varepsilon)}= \psi^{*}_{\varepsilon}(F|_{X}, \nabla^{F}|_{X}).
\end{aligned}\end{align}
 We \textbf{assume} that under the identification (\ref{e.760}), we have
\begin{align}\begin{aligned}\label{e.361}
h^{F}|_{U_{\varepsilon}}= \psi^{*}_{\varepsilon}(h^{F}|_{X}).
\end{aligned}\end{align}
If $h^{F}$ is flat, i.e., $\nabla^{F}h^{F}=0$, then (\ref{e.361}) is a consequence of the flatness of $h^{F}$.\\

In all of this paper, we \textbf{assume} that the triple $(T^{H}M, g^{TZ}, h^{F})$ has the product structures on $U_{\varepsilon}$, i.e.,
\begin{align}\begin{aligned}\label{e.761}
\text{ (\ref{e.759}), (\ref{e.360}) and (\ref{e.361}) hold.}
\end{aligned}\end{align}
We impose the absolute boundary conditions on $(M_{1},X)$ and the relative boundary conditions on $(M_{2},X)$.
We denote the corresponding torsion forms by $\mathscr{T}_{\rm abs}(T^{H}M_{1}, g^{TZ_{1}}, h^{F})$ and $\mathscr{T}_{\rm rel}(T^{H}M_{2}, g^{TZ_{2}}, h^{F})$ respectively (see Def. \ref{d.6}).

Let $\big(H^{p}(Z_{1}, F),\, \nabla^{H^{p}(Z_{1}, F)}\big)$ denote the flat vector bundle on $S$ whose fiber is isomorphic to the absolute cohomology group
$H^{p}(Z_{1, b}, F)$ at $b\in S$. Let $\big(H^{p}(Z_{2}, Y, F),\, \nabla^{H^{p}(Z_{2}, Y, F)}\big)$ denote the flat vector bundle on $S$ whose fiber is isomorphic to the relative cohomology
group $H^{p}(Z_{2, b}, Y_{b}, F)$ at $b\in S$. Then we have a long exact sequence $(\mathscr{H}, \delta)$ of flat vector bundles of cohomology groups (cf. \cite[(0.16)]{BruMa12}), i.e.,
\begin{align}\begin{aligned}\label{e.434}
\cdots \longrightarrow H^{p}(Z, F)\overset{\delta}{\longrightarrow} H^{p}(Z_{1}, F) \overset{\delta}{\longrightarrow}
H^{p+1}(Z_{2}, Y, F) \overset{\delta}{\longrightarrow}\cdots.
\end{aligned}\end{align}
The $\mathbb{Z}-$grading of (\ref{e.434}) on $H^{p}(Z,F)$, $H^{p}(Z_{1},F)$ and $H^{p}(Z_{2},Y,F)$ are given by $3p+1$, $3p+2$ and $3p$ respectively with $0\leq p \leq m$. We denote the $L^{2}-$metric on $\mathscr{H}$ by $h^{\mathscr{H}}_{L^{2}}$
and the canonical flat connection by $\nabla^{\mathscr{H}}$.
Then we associate a torsion form $T_{f}(A^{\mathscr{H}}, h_{L^{2}}^{\mathscr{H}})$
to the triple $(\mathscr{H},  A^{\mathscr{H}}:=\delta+\nabla^{\mathscr{H}},h^{\mathscr{H}}_{L^{2}})$
(see Def. \ref{d.9}) for $f(x)=xe^{x^{2}}$, which verifies the following equation
\begin{align}\begin{aligned}\label{e.482}
dT_{f}(A^{\mathscr{H}}, h_{L^{2}}^{\mathscr{H}})=\sum_{p=0}^{m}(-1)^{p}&\left[f\big(\nabla^{H^{p}(Z_{2},Y,F)},h^{H^{p}(Z_{2},Y,F)}_{L^{2}}\big)\right.\\
&\left.-f\big(\nabla^{H^{p}(Z,F)},h^{H^{p}(Z,F)}_{L^{2}}\big)+f\big(\nabla^{H^{p}(Z_{1},F)},h^{H^{p}(Z_{1},F)}_{L^{2}}\big)\right].
\end{aligned}\end{align}

Let $Q^{S}$ be the vector space of real even forms on $S$ and $Q^{S, 0}$ be the vector space of real exact even forms on $S$.
Let $\chi(Y)$ be the Euler characteristic of $Y$.

We formulate a conjecture
about the general gluing formula of analytic torsion forms in order to answer the open problem proposed in the conference  \cite{Conf03}
on higher torsion
invariants at G\"{o}ttingen 2003.
\begin{conj}\label{c.1}
With the assumption of product structures (\ref{e.761}), the following identity holds in $Q^{S}/Q^{S, 0}$
\begin{align}\begin{aligned}\label{e.474}
\mathscr{T}(T^{H}M, g^{TZ}, h^{F})-\mathscr{T}_{\rm abs}(T^{H}M_{1}, g^{TZ_{1}}, h^{F})&-\mathscr{T}_{\rm rel}(T^{H}M_{2}, g^{TZ_{2}}, h^{F})\\
&=\frac{\log 2}{2}\rk(F)\chi(Y)+T_{f}(A^{\mathscr{H}}, h_{L^{2}}^{\mathscr{H}}).
\end{aligned}\end{align}
\end{conj}
The $0-$degree component of (\ref{e.474}) is exactly the gluing formula of Br\"{u}ning and Ma \cite[(0.22)]{BruMa12} in the case with product structures.  \\

The main result of this paper is the following theorem:

\begin{thm}\label{t.14} Under the assumption {\rm(\ref{e.761})} with the existence of fiberwise well-defined Morse function
{\rm (see (\ref{e.634}), (\ref{e.636}), (\ref{e.650}), (\ref{e.635}))}, for the fibration $M\overset{\pi}\rightarrow S$,
the following identity holds in $Q^{S}/Q^{S, 0}$
\begin{align}\begin{aligned}\label{e.355}
\mathscr{T}(T^{H}M, g^{TZ}, h^{F})-\mathscr{T}_{\rm abs}(T^{H}M_{1}, g^{TZ_{1}}, h^{F})&-\mathscr{T}_{\rm rel}(T^{H}M_{2}, g^{TZ_{2}}, h^{F})\\
&=\frac{\log 2}{2}\rk(F)\chi(Y)+T_{f}(A^{\mathscr{H}}, h_{L^{2}}^{\mathscr{H}}).
\end{aligned}\end{align}
\end{thm}

The ideal that we use to prove Theorem \ref{t.14} is analogous to that Br\"{u}ning and Ma have used in their recent work \cite{BruMa12}, in which they have proved the gluing formula of Ray-Singer
 analytic torsion for flat vector bundles in full generality. In our proof of Theorem \ref{t.14}, the main tools are Bismut-Goette's
equivariant family extension \cite{BGo} of Bismut-Zhang's results \cite{BisZh92}, \cite{BisZh94} and a result of Goette
 \cite[Thm. 7.37]{Goette03}, by using a technique due to Ma \cite{Ma99}, \cite{Ma00}, \cite{Ma02} to treat analytic torsion form in a different
context. The defect of Theorem \ref{t.14} is that the fiberwise well-defined Morse function does not always exist because of some topological
 obstructions of the fibration (cf. \cite[Thm. 5.9]{BGo}). In the author's thesis \cite{Zhu13}, Conjecture \ref{c.1} was also established by using a different technique if $H(Y,F)=0$.

 The whole paper is organized as follows. In Section \ref{s.4},
  we introduce the definition of Bismut-Lott torsion form for fibrations with boundary. We show the corresponding version of the formula
(\ref{e.358}) in the case with boundary. In Section \ref{s.2}, we recall Morse theory in the family case when there exists a fiberwise Morse function, then establish the double formulas of torsion form. In Section \ref{s.3}, we prove Theorem \ref{t.14} under the assumption of existence of a fiberwise Morse
 function.

\textbf{Acknowledgments.}
This paper is a part of the author's Ph.D. thesis at Universit{\'e} Paris Diderot-Paris VII under the direction of Professor Xiaonan Ma.
He would like to thank Professor Xiaonan Ma for his patient instruction and constant encouragement.

\section{Bismut-Lott torsion form in the case with boundary}\label{s.4}

In this section, we will introduce the geometric background and define Bismut-Lott's torsion form with natural boundary conditions for flat vector bundles over a smooth fibration $\pi:M\rightarrow S$ with boundary. This definition of analytic torsion form was introduced by
Bismut and Lott in \cite{BL} for fibrations without boundary.

This section is organized as follows.
In Section \ref{ss1.3}, we recall the torsion form $T_{f}(A', h^{E})$ associated to a complex of
flat vector bundles (cf. \cite[$\S$2]{BL}).
In Section \ref{ss1.1},  we introduce some geometric concepts about smooth
fibrations with boundary. Then we explain our assumptions of product structures near $X$. In Section \ref{ss1.4}, we introduce the
Bismut-Lott superconnection and describe
 the natural boundary conditions. In Section \ref{ss1.5}, we establish the double formula of heat kernel in the family case.
 In Section \ref{ss1.7}, we extend
 the definition of analytic torsion form to the boundary case. We also establish the differential form version of
``Riemann-Roch-Grothendieck"
 theorem of Bismut-Lott in the case with boundary.

\subsection{Torsion form associated to the complex of flat vector bundles}\label{ss1.3}
 Let $S$ be a compact smooth manifold of dimension $n$. Let $TS$ be the tangent bundle of $S$ and $T^{*}S$ be the cotangent bundle. For a vector bundle $F$ on $S$, let $\Omega^{j}(S,F)$ be the space of $F$-valued smooth differential $j-$forms on $S$, $\Omega(S,F)=\bigoplus_{j=0}^{n}\Omega^{j}(S,F)$  and $\Omega^{\bullet}(S)=\Omega^{\bullet}(S,\mathbb{R})$.

 Let $E=E_{+}\oplus E_{-}$ be a $\mathbb{Z}_{2}$-graded complex vector bundle over $S$ with a flat connection $\nabla^{E}=\nabla^{E_{+}}\oplus\nabla^{E_{-}}$, i.e., the curvature $(\nabla^{E_{\pm}})^{2}$ is zero. By definition, a Hermitian metric $h^{E}$ on $\mathbb{Z}_{2}$-graded bundle $E$ is a Hermitian metric such that $E_{+}$ and $E_{-}$ are orthogonal.

Let $(\nabla^{E})^{*}$ be the adjoint of $\nabla^{E}$ with respect to $h^{E}$. Let
\begin{align}\begin{aligned}\label{e.499}
\omega(E, h^{E})=(\nabla^{E})^{*}-\nabla^{E}=(h^{E})^{-1}\nabla^{E}h^{E}\in \Omega^{1}(S, \text{End}(E)).
\end{aligned}\end{align}\index{$\omega(E, h^{E})$}
Let $\varphi:\Omega(S)\rightarrow \Omega(S)$ be the linear map such that for all $\beta\in \Omega^{k}(S)$,
\begin{align}\begin{aligned}\label{e.500}
\varphi\beta=(2i\pi)^{-k/2}\beta.
\end{aligned}\end{align}

In this paper, we always set
\begin{align}\begin{aligned}\label{e.762}
f(a)=a \exp(a^{2}),
\end{aligned}\end{align}
 which is a homomorphic odd function over $\mathbb{C}$.
\begin{defn}\label{d.23}
Put
\begin{align}\begin{aligned}\label{e.501}
f(\nabla^{E}, h^{E})=(2i\pi)^{1/2}\varphi\tr_{s}\left[f(\frac{\omega(E,h^{E})}{2})\right]\in \Omega(S),
\end{aligned}\end{align}
where $\tr_{s}[\cdot,\cdot]:=\tr|_{E_{+}}-\tr|_{E_{-}}$ denotes the supertrace (cf. \cite{BGV92}). It is a real, odd and closed form and its de Rham cohomology class does not depend on the choice of $h^{E}$ (cf. \cite[Theorems 1.8, 1.11]{BL}).
\end{defn}

\begin{defn}\label{d.8}
Let $h^{'E}$ be another Hermitian metric on $E$. As \cite[Def. 1.12]{BL}, we define
\begin{align}\begin{aligned}\label{e.51}
\widetilde{f}(\nabla^{E}, h^{E}, h^{'E})=\int_{0}^{1}\varphi\tr_{s}\left[\frac{1}{2}(h^{E}_{l})^{-1}\frac{\partial h^{E}_{l}}{\partial l}f'\big(\frac{\omega(E, h^{E}_{l})}{2}\big)\right]dl\in Q^{S}/Q^{S, 0},
\end{aligned}\end{align}
where $h^{E}_{l}, l\in[0, 1]$ is a smooth path of metrics on $E$ such that $h^{E}_{0}=h^{E}$ and $h^{E}_{1}=h^{'E}$.
\end{defn}

Then from \cite[Thm. 1.11]{BL}, we get
\begin{align}\begin{aligned}\label{e.52}
d\widetilde{f}(\nabla^{E}, h^{E}, h^{'E})=f(\nabla^{E}, h^{'E})-f(\nabla^{E}, h^{E}).
\end{aligned}\end{align}
Moreover,
the class $\widetilde{f}(\nabla^{E}, h^{E}, h^{'E})\in Q^{S}/Q^{S, 0}$ does not depend on the choice of the path  $h^{E}_{l}$.

Let
\begin{align}\begin{aligned}\label{5.45}
(E,\nabla^{E}, v):0\rightarrow E^{0}\overset{v}\rightarrow E^{1}\overset{v}\rightarrow \cdots \overset{v}\rightarrow E^{k}\rightarrow 0
\end{aligned}\end{align}
be a flat complex of complex vector bundles on $S$. That is,
$$\nabla^{E}=\bigoplus_{i=0}^{k}\nabla^{E^{i}}$$
is a flat connection on $E=\bigoplus_{i=0}^{k}E^{i}$ and $v$ is a flat chain map $(E,\nabla^{E},v)$, meaning
\begin{align}\begin{aligned}\label{5.47}
(\nabla^{E})^{2}=0, \quad v^{2}=0, \quad [\nabla^{E}, v]=0.
\end{aligned}\end{align}
Then $A'=v +\nabla^{E}$
gives a flat superconnection of total degree 1 on $E$. By \cite[$\S$2(a)]{BL}, the cohomology $H(E)$ of the complex (\ref{5.45}) is a flat vector bundle on $S$, and
let $\nabla^{H(E)}$ be the flat connection on $H(E)$ induced by $\nabla^{E}$.  Let $h^{H(E)}$ be the Hermitian metric on $H(E)$ induced by $h^{E}$ by finite dimensional Hodge theory.

Let $N\in \text{End}(E)$ be the number operator of $E$, i.e., $N$ acts on $E^{i}$ by multiplication by $i$.
\begin{defn}\label{d.40} For $t>0$, put
\begin{align}\begin{aligned}\label{e.778}
h^{E}_{t}=\bigoplus_{i=0}^{k}t^{i}h^{E^{i}}.
\end{aligned}\end{align}
\end{defn}
Then $h^{E}_{t}$ is a metric on $E$ and $h^{E}=h^{E}_{1}$. Let $A''$ be the adjoint of $A'$ with respect to $h^{E}$ and $A''_{t}$ be the adjoint of $A'$ with respect to $h^{E}_{t}$, then we have
\begin{align}\begin{aligned}\label{e.779}
A''_{t}=t^{-N}A''t^{N}.
\end{aligned}\end{align}

Following the formalism of \cite[$\S$1(e)]{BL}, we put
\begin{align}\begin{aligned}\label{e.777}
X_{t}=\frac{1}{2}(A''_{t}-A').
\end{aligned}\end{align}
 Let $f'$ be the derivative of $f$, we put
\begin{align}\begin{aligned}\label{e.769}
&f^{\wedge}(A',h^{E}_{t})=\varphi\tr_{s}\left[\frac{N}{2}f'(X_{t})\right]\in \Omega(S),\\
&d(E)=\sum_{i=0}^{k}(-1)^{i}i\rk(E^{i}), \quad d(H(E))=\sum_{i=0}^{k}(-1)^{i}i\rk(H^{i}(E)).
\end{aligned}\end{align}

\begin{defn}\label{d.9}
Following \cite[Def. 2.20]{BL}, the torsion form $T_{f}(A', h^{E})\in \Omega(S)$ associated to (\ref{5.45}) is defined by
\begin{align}\begin{aligned}\label{e.771}
T_{f}(A',h^{E})=-\int_{0}^{+\infty}\Big[f^{\wedge}&(A',h^{E}_{t})-d(H(E))\frac{f'(0)}{2}\\
&-[d(E)-d(H(E))]\frac{f'(\frac{i\sqrt{t}}{2})}{2}\Big]\frac{dt}{t}.
\end{aligned}\end{align}
\end{defn}

Then by \cite[Thm. 2.22]{BL}, we have

\begin{thm}\label{t.21}
The form $T_{f}(A', h^{E})$ is even and real. Moreover
\begin{align}\begin{aligned}\label{e.526}
dT_{f}(A', h^{E})=f(\nabla^{E}, h^{E})-f(\nabla^{H(E)}, h^{H(E)}).
\end{aligned}\end{align}
\end{thm}

Let $h^{E}_{0}, h^{E}_{1}$ be two different metrics on $E$, we denote by $h^{H(E)}_{i},\,i=0,1,$ the metrics on $H(E)$ induced by $h^{E}_{i}$ on $E$ as
above. We set
\begin{align}\begin{aligned}\label{e.748}
&\widetilde{f}(\nabla^{E}, h^{E}_{0}, h^{E}_{1}):=\sum_{i=0}^{k}(-1)^{i}\widetilde{f}(\nabla^{E^{i}}, h^{E^{i}}_{0}, h^{E^{i}}_{1}),\\
&\widetilde{f}(\nabla^{H(E)}, h^{H(E)}_{0}, h^{H(E)}_{1}):=\sum_{i=0}^{k}(-1)^{i}\widetilde{f}(\nabla^{H^{i}(E)}, h^{H^{i}(E)}_{0}, h^{H^{i}(E)}_{1}).
\end{aligned}\end{align}
The following result is \cite[Thm. 2.24]{BL}.
\begin{thm}\label{t.16}The following identity holds in $Q^{S}/Q^{S, 0}$
\begin{align}\begin{aligned}\label{e.391}
T_{f}(A', h^{E}_{1})-T_{f}(A', h^{E}_{0})=\widetilde{f}(\nabla^{E}, h^{E}_{0}, h^{E}_{1})-\widetilde{f}(\nabla^{H(E)}, h^{H(E)}_{0}, h^{H(E)}_{1}).
\end{aligned}\end{align}
\end{thm}

Let
\begin{align}\begin{aligned}\label{5.51}
(E, v)&:0\rightarrow E^{0}\overset{v}\rightarrow E^{1}\overset{v}\rightarrow \cdots \overset{v}\rightarrow E^{l}\rightarrow 0, \\
(E', u)&:0\rightarrow E^{l}\overset{u}\rightarrow E^{l+1}\overset{u}\rightarrow \cdots \overset{u}\rightarrow E^{l+m'}\rightarrow 0,
\end{aligned}\end{align}
be two exact sequences of flat vector bundles over $S$, in which $E^{l}$ is the same flat vector bundle.
We construct another exact sequence of flat vector bundles from (\ref{5.51}), that is
\begin{align}\begin{aligned}\label{5.52}
(E'\circ E, v):0\rightarrow E^{0}\overset{v}\rightarrow E^{1}\overset{v}\rightarrow \cdots \overset{v}\rightarrow E^{l-1}
\overset{u\circ v}\rightarrow E^{l+1}\overset{u}\rightarrow\cdots \overset{u}\rightarrow E^{l+m'}\rightarrow 0.
\end{aligned}\end{align}
Then following the proof of \cite[Thm. 1.22]{BGS1},
 we get the relation between the torsion forms of the above three exact sequences by using \cite[Thm. A 1.4]{BL}.
\begin{lemma}\label{l.36}In $Q^{S}/Q^{S, 0}$, we have
\begin{align}\begin{aligned}\label{5.53}
T_{f}(A', h^{E'\circ E})=T_{f}(A', h^{E})+(-1)^{l+1}T_{f}(A', h^{E'}).\end{aligned}\end{align}
\end{lemma}
Now we calculate explicitly this torsion form of Definition \ref{d.9} in a simple case. And the results thus obtained will have applications in Section \ref{s.3}.

Now let $(F,\nabla^{F})$ be a flat complex vector bundle over $S$. let $h^{F}$ be the Hermitian metric on $F$. Let $\tau\in \text{Aut}(F)$ be a locally constant diagonalizable automorphism of $F$, i.e., for any $x_{0}\in S$ there exists a neighborhood $U_{x_{0}}$ of $x_{0}$ such that there exists a local frame $\{f_{j}\}_{j=1}^{\rk(F)}$ of $F$ on $U_{x_{0}}$, with respect to which $\tau|_{U_{x_{0}}}$ is a constant diagonal matrix.

Set $(E^{0},\nabla^{E^{0}},h^{E^{0}})=(F,\nabla^{F},h^{F})$ and $(E^{1},\nabla^{E^{1}},h^{E^{1}})=(F,\nabla^{F},h^{F})$, we get the following short exact sequence of flat vector bundles verifying (\ref{5.47}):
\begin{align}\begin{aligned}\label{e.780}
(E,\nabla^{E},\tau): 0\rightarrow E^{0}\overset{\tau}\rightarrow E^{1}\rightarrow 0.
\end{aligned}\end{align}

Similar to the proof of \cite[Thm A1.1]{BL}, now we will calculate the torsion form $T_{f}(A',h^{E})$ associated to (\ref{e.780}), which contains only the zero degree component.

\begin{lemma}\label{l.49} For the flat short exact sequence of complex vector bundles {\rm(\ref{e.780})}, we have
\begin{align}\begin{aligned}\label{e.781}
T_{f}(A',h^{E})=-\log |\det \tau|.
\end{aligned}\end{align}
\end{lemma}
\begin{proof}
First, we note that both sides of the equation (\ref{e.781}) are globally well-defined, so we just need to verify it locally. For any $x_{0}\in S$, by our assumption on $\tau\in \text{Aut}(F)$ there exists a neighborhood $U_{x_{0}}$ of $x_{0}$ such that on $U_{x_{0}}$ we have
\begin{align}\begin{aligned}\label{e.782}
[\nabla^{E},\tau]=0,\quad [(\nabla^{E})^{*},\tau]=0,\quad [\nabla^{E},\tau^{*}]=0.
\end{aligned}\end{align}
Let $r=\rk(F)$, $\tau=\text{diag}(\lambda_{1},\cdots,\lambda_{r})$ on $U_{x_{0}}$ and $\Delta=\tau\tau^{*}+\tau^{*}\tau$. By (\ref{e.782}), we obtain
\begin{align}\begin{aligned}\label{e.783}
&[\omega(E,h^{E}),\tau]=0,\quad [\omega(E,h^{E}),\tau^{*}]=0,\\
&\Delta^{(0)}:=\tau^{*}\tau=\text{diag}(|\lambda_{1}|^{2},\cdots,|\lambda_{r}|^{2}),\\
&\Delta^{(1)}:=\tau\tau^{*}=\text{diag}(|\lambda_{1}|^{2},\cdots,|\lambda_{r}|^{2}).
\end{aligned}\end{align}
In this case by (\ref{e.777}), we find
\begin{align}\begin{aligned}\label{e.784}
X_{t}=\frac{1}{2}(\omega(E,h^{E})+t\tau^{*}-\tau).
\end{aligned}\end{align}
By (\ref{e.783}) and (\ref{e.784}) we obtain
\begin{align}\begin{aligned}\label{e.785}
X_{t}^{2}=\frac{1}{4}(\omega^{2}(E,h^{E})-t\Delta).
\end{aligned}\end{align}
As $f'(a)$ is an even function, there is a holomorphic function $g(a)$ such that $f'(a)=g(a^{2})$. Then by (\ref{e.769}) we have
\begin{align}\begin{aligned}\label{e.786}
f^{\wedge}(A',h^{E}_{t})&=\varphi \tr_{s}\left[\frac{N}{2}g(X^{2}_{t})\right]\\
&=\sum_{i=0}^{1}(-1)^{i}i\varphi\tr\Big[\frac{1}{2}g\big(\frac{1}{4}\omega^{2}(E,h^{E})-t\frac{\Delta^{(i)}}{4}\big)\Big].
\end{aligned}\end{align}
By \cite[Prop. 1.3]{BL}, the even degree part of $f^{\wedge}(A',h^{E}_{t})$ must disappear except the $0$-degree, hence by (\ref{e.783}) and (\ref{e.786}) we obtain
\begin{align}\begin{aligned}\label{e.787}
f^{\wedge}(A',h^{E}_{t})=-\tr\Big[\frac{1}{2}g\big(-t\frac{\Delta^{(1)}}{4}\big)\Big]=-\sum_{j=1}^{r}\frac{1}{2}f'\Big(\frac{i\sqrt{t}|\lambda_{j}|}{2}\Big).
\end{aligned}\end{align}
In this case by (\ref{e.769}), we get that $d(E)=-r$ and $d(H(E))=0$, thus from Definition \ref{d.9} we have
\begin{align}\begin{aligned}\label{e.788}
T_{f}(A',h^{E})&=\int_{0}^{+\infty}\Big[\sum_{j=1}^{r}\frac{1}{2}f'\Big(\frac{i\sqrt{t}|\lambda_{j}|}{2}\Big)-r\frac{f'(\frac{i\sqrt{t}}{2})}{2}\Big]\frac{dt}{t}\\
&=\frac{1}{2}\sum_{j=1}^{r}\int_{0}^{+\infty}\Big[g\Big(-\frac{t|\lambda_{j}|^{2}}{4}\Big)-g(-\frac{t}{4})\Big]\frac{dt}{t}\\
&=-\sum_{j=1}^{r}\log(|\lambda_{j}|)=-\log|\det\tau|.
\end{aligned}\end{align}
Then (\ref{e.781}) follows from (\ref{e.788}). The proof of Lemma \ref{l.49} is completed.
\end{proof}

\subsection{Smooth fibration with boundary and product structures}\label{ss1.1}

Let
 $\pi:M\rightarrow S$ be a smooth fibration with boundary $X:=\partial M$ whose standard fiber $Z$ is a compact manifold
with boundary $Y:=\partial Z$ (see Figure \ref{fig.1}) and  $\dim Z=m$. We assume that the boundary $X$ of $M$ is a smooth fibration over $S$ denoted
by $\pi_{\partial}:X\rightarrow S$, whose standard fiber is the compact smooth manifold $Y$.

\begin{figure}
\input{fibration_left.pstex_t}
\caption[Figure]{\label{fig.1}}
\end{figure}\index{Figure \ref{fig.1}}

 Let $TZ\subset TM$ be the vertical tangent bundle of $M$ and
 $T^{*}Z$ be its dual bundle. Let $TY\subset TX$ be the vertical tangent bundle of the fibration $\pi_{\partial}:X\rightarrow S$.
We see that $TY$ is a subbundle of $TZ$ restricted on $X$. Let $N$ be the normal bundle of $X\subset M$, i.e., $N:=TM/TX$, then
we have $(TZ/TY)|_{X}\cong N$. We note that $N$ is a trivial line bundle on $X$ (cf. \cite[p.54, p.66]{BottTu}), one way to show this
 is that the inward-pointing normal vector field gives a global frame of $N$ on $X$.

  Let $T^{H}M\subset TM$ be a horizontal bundle of $M$, such that $TM=T^{H}M\oplus TZ$, then we have
 $T^{H}M\cong \pi^{*}TS$. This induces the isomorphism
 \begin{align}\begin{aligned}\label{e.582}
\Lambda(T^{*}M)\cong \pi^{*}(\Lambda(T^{*}S))\otimes \Lambda(T^{*}Z)
\end{aligned}\end{align}
as bundles of $\mathbb{Z}-$graded algebras over $M$.

\begin{defn}
Let $X$ be a compact manifold and $I$ (not be a point) be an interval of $\mathbb{R}$, we set $X_{I}:=X\times I$, for example $X_{[-R,R]}=X\times[-R,R]$, $X_{\mathbb{R}}=X\times(-\infty,+\infty)$.
\end{defn}

Let $g^{TZ}$ be a metric on $TZ$, let $g^{T Y}$ be that of $TY$ induced by $g^{TZ}$.
Let $X_{(-\varepsilon, 0]}$ be a product neighborhood of the boundary $X$.
 Let $\psi_{\varepsilon}$ denote the projection map on the first factor:
\begin{align}\begin{aligned}\label{e.3}
X _{(-\varepsilon, 0]}\ni (x', x_{m})\overset{\psi_{\varepsilon}} \longrightarrow x'\in X.
\end{aligned}\end{align}
 \textbf{We assume that} $T^{H}M$ and $g^{TZ}$ have product structures on $X _{(-\varepsilon, 0]}$, i.e.,
\begin{align}\begin{aligned}\label{e.598}
(T^{H}M)|_{X}\subset TX,\quad \big(T^{H}M\big)|_{X _{(-\varepsilon, 0]}}=\psi_{\varepsilon}^{*}\big((T^{H}M)|_{X}\big),
\end{aligned}\end{align}
\begin{align}\begin{aligned}\label{e.4}
g^{TZ}(x', x_{m})=g^{T Y}(x')+ dx^{2}_{m}, \quad (x', x_{m})\in X_{(-\varepsilon, 0]}.
\end{aligned}\end{align}
Then $T^{H}X:=(T^{H}M)|_{X}$ gives a horizontal bundle of the fibration $\pi_{\partial}:X\rightarrow S$, such that $TX=T^{H}X\oplus TY$.
 \begin{rem}\label{r.7}
 We can always chose a horizontal bundle $T^{H}M$ of $M$ such that our assumption (\ref{e.598}) is true. For example, let $g^{TM}$ be a Riemannian metric such that for some $\varepsilon'>\varepsilon>0$
 \begin{align}\begin{aligned}\label{e.595}
g^{TM}(x', x_{m})=g^{T X}(x')\oplus dx^{2}_{m}, \quad \text{ for }(x', x_{m})\in X_{(-\varepsilon', 0]}.
\end{aligned}\end{align}
 Then we chose $T^{H}M$ to be the orthogonal complement of $TZ$ in $(TM, g^{TM})$ which satisfies our assumption (\ref{e.598}). \end{rem}

Let $(F, \nabla^{F})$ be a flat complex vector bundle on $M$. Let $h^{F}$ be a Hermitian metric on $F$.
On $X _{(-\varepsilon, 0]}$ we trivialize $F$ by parallel transport with respect to $\nabla^{F}$ along the path
$$\gamma_{x'}(t):t\in (-\varepsilon, 0]\rightarrow (x', tx_{m})\in X _{(-\varepsilon, 0]}, \, x' \in X, $$
then it follows from the flatness of $\nabla^{F}$ that on $X _{(-\varepsilon, 0]}$
\begin{align}\begin{aligned}\label{e.5}
\quad (F, \nabla^{F})|_{X _{(-\varepsilon, 0]}}= \psi^{*}_{\varepsilon}(F|_{X}, \nabla^{F}|_{X}).
\end{aligned}\end{align}

 \textbf{We assume that} the Hermitian metric $h^{F}$ on $X _{(-\varepsilon, 0]}$ is the pull-back of its restriction on $X$, i.e., under the trivialization (\ref{e.5}) of $F$, we have
\begin{align}\begin{aligned}\label{e.6}
h^{F}|_{X _{(-\varepsilon, 0]}}= \psi^{*}_{\varepsilon}(h^{F}|_{X}).
\end{aligned}\end{align}
 If $h^{F}$ is flat, i.e., $\nabla^{F}h^{F}=0$, then (\ref{e.6}) is a consequence of the flatness of $h^{F}$.
We mention that these assumptions (\ref{e.4}), (\ref{e.6})
of product structures are formulated in \cite[$\S$2.1]{BruMa12}

\subsection{Bismut-Lott superconnection and boundary conditions}\label{ss1.4}

Let $\Omega^{\bullet}(Z,F|_{Z})$ be the infinite-dimensional $\mathbb{Z}-$graded vector bundle over $S$ whose fiber is $\Omega^{\bullet}(Z_{b},F|_{Z_{b}})$ at $b\in S$. That is
\begin{align}\begin{aligned}\label{e.540}
\Omega^{\bullet}(M,F)=\Omega^{\bullet}(S,\Omega^{\bullet}(Z,F|{Z})).
\end{aligned}\end{align}

  Let $o(TZ)$ be the orientation bundle of $TZ$ (cf. \cite[p.88]{BottTu}),
 which is a flat real line bundle on $M$.
Let $dv_{Z}$ be the Riemannian volume form on fibers $Z$ associated to $g^{TZ}$, which is a section of $\Lambda^{m}(T^{*}Z)\otimes o(TZ)$ over $M$.
 Let $g^{\Lambda(T^{*}Z)\otimes F}$ be the metric on $\Lambda(T^{*}Z)\otimes F$ induced by $g^{TZ}$ and $h^{F}$, then it induces a Hermitian metric on $\Omega^{\bullet}(Z,F|_{Z})$ defined by:
for $s, s'\in \Omega^{\bullet}(Z_{b},F|_{Z_{b}})$, $b\in S$,
\begin{align}\begin{aligned}\label{e.547}
\langle s, s'\rangle_{h^{\Omega^{\bullet}(Z,F|_{Z})}}(b):=\int_{Z_{b}}\langle s, s' \rangle_{g^{\Lambda(T^{*}Z)\otimes F}}(x)dv_{Z_{b}}(x).
\end{aligned}\end{align}

 Let $P^{TZ}$ denote the projection from $TM=T^{H}M\oplus TZ$ to $TZ$. For $U\in TS$, let
 $U^{H}$ be the horizontal lift of $U$ in $T^{H}M$, so that $\pi_{*}U^{H}=U$.

\begin{defn}\label{d.32}
For $s\in C^{\infty}(S, \Omega^{\bullet}(Z,F|_{Z}))$ and $U\in TS$, the Lie derivative $L_{U^{H}}$ acts on $C^{\infty}(S,\Omega^{\bullet}(Z,F|_{Z}))$. Then
\begin{align}\begin{aligned}\label{e.541}
\nabla^{\Omega^{\bullet}(Z,F|_{Z})}_{U}s:=L_{U^{H}}s
\end{aligned}\end{align}\index{$\nabla^{W}$}
defines a connection on $\Omega^{\bullet}(Z,F|_{Z})$ preserving the $\mathbb{Z}-$grading.
\end{defn}

 Let $d^{Z}$ be the exterior differentiation along fibers $(Z, F,\nabla^{F})$.
If $U_{1}, U_{2}$ are two vector fields on $S$, put
\begin{align}\begin{aligned}\label{e.542}
T(U_{1}, U_{2})=-P^{TZ}[U^{H}_{1}, U^{H}_{2}]\in C^{\infty}(M, TZ),
\end{aligned}\end{align}
then $T$ is a tensor, i.e., $T\in C^{\infty}(M, \pi^{*}\Lambda^{2}(T^{*}S)\otimes TZ)$. Let $i_{T}$ be the interior multiplication in the vertical direction by $T$.

The flat connection $\nabla^{F}$ extends naturally
 to be an exterior differential operator $d^{M}$ acting on $\Omega^{\bullet}(M, F)$,
then it defines a flat superconnection of total degree 1 on $\Omega^{\bullet}(Z,F|_{Z})$.
By \cite[Prop. 3.4]{BL}, we have the following identity
\begin{align}\begin{aligned}\label{e.544}
d^{M}=d^{Z}+\nabla^{\Omega^{\bullet}(Z,F|_{Z})}+i_{T}.
\end{aligned}\end{align}

Let $(\nabla^{\Omega^{\bullet}(Z,F|_{Z})})^{*}, \, (d^{M})^{*}, \, (i_{T})^{*}$, $(d^{Z})^{*}$ be the formal adjoints
 of $\nabla^{\Omega^{\bullet}(Z,F|_{Z})}, \, d^{M}, \, i_{T}$, $d^{Z}$ with respect to the Hermitian metric $h^{\Omega^{\bullet}(Z,F|_{Z})}$ in (\ref{e.547}).
Set
\begin{align}\begin{aligned}\label{e.548}
&D^{Z}=d^{Z}+(d^{Z})^{*}, \quad \nabla^{\Omega^{\bullet}(Z,F|_{Z}), u}=\frac{1}{2}(\nabla^{\Omega^{\bullet}(Z,F|_{Z})}+(\nabla^{\Omega^{\bullet}(Z,F|_{Z})})^{*}).
\end{aligned}\end{align}
Then the Hodge Laplacian associated to $g^{TZ}$ and $h^{F}$ along the fibers $Z$ is given by
\begin{align}\begin{aligned}\label{e.485}
(D^{Z})^{2}=d^{Z}(d^{Z})^{*}+(d^{Z})^{*}d^{Z}:\Omega^{\bullet}(Z, F|_{Z})\rightarrow \Omega^{\bullet}(Z, F|_{Z}).
\end{aligned}\end{align}

Let $N$\index{$N$} be the number operator on $\Omega^{\bullet}(Z, F|_{Z})$, i.e., it acts by multiplication by $k$ on $\Omega^{k}(Z, F|_{Z})$.
For $t>0$, we set
\begin{align}\begin{aligned}\label{e.549}
&C'_{t}=t^{N/2}d^{M}t^{-N/2}, \quad C''_{t}=t^{-N/2}(d^{M})^{*}t^{N/2}, \\
&C_{t}=\frac{1}{2}(C'_{t}+C''_{t}), \quad\quad D_{t}=\frac{1}{2}(C''_{t}-C'_{t}).
\end{aligned}\end{align}
Then $C''_{t}$ is the adjoint of $C'_{t}$ with respect to $h^{\Omega^{\bullet}(Z, F|_{Z})}$. We note that $C_{t}$ is a superconnection and $D_{t}$ is an odd element of $\Omega(S, \text{End}(\Omega^{\bullet}(Z, F|_{Z})))$. Moreover, we have
\begin{align}\begin{aligned}\label{e.550}
C^{2}_{t}=-D^{2}_{t}.
\end{aligned}\end{align}

Let $g^{TS}$ be a Riemannian metric on $TS$, then $g^{TM}=\pi^{*}g^{TS}\oplus g^{TZ}$ defines a Riemannian metric on
 $TM=T^{H}M\oplus TZ$. Let $\nabla^{TM}, \nabla^{TS}$ denote the corresponding Levi-Civita connections on $TM$ and $TS$. Then
\begin{align}\begin{aligned}\label{e.551}
\nabla^{TZ}=P^{TZ}\nabla^{TM}
\end{aligned}\end{align}
defines the canonical connection on $TZ$, which is independent of the choice of $g^{TS}$ (cf. \cite[Def. 1.6, Thm. 1.9]{Bismut86}).

For $X\in TZ$, let $X^{*}\in T^{*}Z$ be the dual of $X$ by the metric $g^{TZ}$. Set
\begin{align}\begin{aligned}\label{e.553}
c(X)=X^{*}\wedge- \, i_{X}, \quad \widehat{c}(X)=X^{*}\wedge+\, i_{X},
\end{aligned}\end{align}
where $i_{\cdot}$ denotes the interior multiplication.

By \cite[Prop. 3.9]{BL}, we get
\begin{align}\begin{aligned}\label{e.556}
C_{t}=\frac{\sqrt{t}}{2}D^{Z}+\nabla^{\Omega^{\bullet}(Z, F|_{Z}), u}-\frac{1}{2\sqrt{t}}c(T),
\end{aligned}\end{align}
which is essentially the same as the Bismut superconnection (cf. \cite[$\S$III.a)]{Bismut86}).

Now we introduce the boundary conditions. Let $e_{\bf{n}}$ be the inward-pointing unit normal vector field on $X$ and $e^{\bf{n}}$ be its
dual vector field, then we extend $e_{\bf{n}}$, $e^{\bf{n}}$ on $X _{(-\varepsilon, 0]}$. By (\ref{e.4}) we have  $e_{\bf{n}}=-\frac{\partial}{\partial x_{m}},\,e^{\bf{n}}=-dx_{m}$ on $X _{(-\varepsilon, 0]}$.
For $\sigma \in \Omega(M,F)$,
 we say that $\sigma$ satisfies the \textbf{absolute boundary conditions} (cf. \cite[Def. 3.2]{RaySing71}, \cite[(1.12)]{BruMa12}), if
\begin{align}\begin{aligned}\label{e.10}
(i_{e_{\bf{n}}}\sigma )|_{X}=(i_{e_{\bf{n}}}d^{Z}\sigma )|_{X}=0.
\end{aligned}\end{align}
 We say that $\sigma$ satisfies the \textbf{relative boundary conditions}, if
\begin{align}\begin{aligned}\label{e.11}
(e^{\bf{n}}\wedge\sigma )|_{X}=(e^{\bf{n}}\wedge (d^{Z})^{*}\sigma) |_{X}=0.
\end{aligned}\end{align}

Set
\begin{align}\begin{aligned}\label{e.579}
\Omega^{\bullet}_{\text{abs}}(Z, F|_{Z})&:=\{\sigma\in \Omega^{\bullet}(Z, F|_{Z})\big|
 \,\sigma \text{ satisfies (\ref{e.10})}\},\\
\Omega^{\bullet}_{\text{rel}}(Z, F|_{Z})&:=\{\sigma\in\Omega^{\bullet}(Z, F|_{Z})\big|
 \,\sigma \text{ satisfies (\ref{e.11})}\}.
\end{aligned}\end{align}
Then they can be regarded as two sub-bundles of $\Omega^{\bullet}(Z, F|_{Z})$.
From now on, we use the subscripts
\begin{align}\begin{aligned}\label{e.612}
\text{``bd''}=\left\{
\begin{array}{ll}
 \text{``abs'', } & \hbox{ the for absolute boundary conditions; }\\
 \text{``rel'', } & \hbox{ the for relative boundary conditions.}
\end{array}
\right.
\end{aligned}\end{align}

 Set
\begin{align}\begin{aligned}\label{e.29}
&(D^{Z})^{2}_{{\rm bd }}=(D^{Z})^{2}|_{\Omega^{\bullet}_{{\rm bd }}(Z, F|_{Z})}.
\end{aligned}\end{align}
 Thus $(D^{Z})^{2}_{{\rm bd }}$ is the operator with the relative or absolute boundary conditions on $X$, moreover
 it is essentially self-adjoint with respect to $h^{\Omega^{\bullet}(Z, F|_{Z})}$.

Let $\Omega^{j}(Z,Y,F|_{Z})=\Omega^{j}(Z,F|_{Z})\bigoplus\Omega^{j-1}(Y,F|_{Y})$. Let $i:X\hookrightarrow M$ denote the inclusion map of $X$ into $M$ (the same notation for the inclusion map of their fibers). The following complex and differential $\big(\Omega(Z,Y,F|_{Z}),d\big)$ computes the relative de Rham cohomology groups along the fiber:
\begin{align}\begin{aligned}\label{e.821}
&d:\Omega^{j}(Z,Y,F|_{Z})\rightarrow \Omega^{j+1}(Z,Y,F|_{Z}),\\
&d(\alpha,\beta):=(d^{X}\alpha,i^{*}\alpha-d^{Y}\beta).
\end{aligned}\end{align}
Then we have for $0\leq j\leq m$
\begin{align}\begin{aligned}\label{e.822}
H^{j}(Z,Y,F)=H^{j}(\Omega^{\bullet}(Z,Y,F|_{Z}),d).
\end{aligned}\end{align}

Now we define $d:\Omega(M,X,F)\rightarrow \Omega(M,X,F)$ by
\begin{align}\begin{aligned}\label{e.823}
d(\alpha,\beta):=(d^{M}\alpha,i^{*}\alpha-d^{X}\beta).
\end{aligned}\end{align}

Let $\pi_{\partial}:X\rightarrow S$. For $\gamma \in \Omega(S)$ and $(\alpha,\beta)\in \Omega(M,X,F)$, to make our sign conventions compatible with \cite[(1.51)]{BruMa06} we define
 \begin{align}\begin{aligned}\label{e.824}
(\pi^{*}\gamma)\wedge(\alpha,\beta):=\big((\pi^{*}\gamma)\wedge\alpha,(-1)^{\deg(\gamma)}(\pi^{*}_{\partial}\gamma)\wedge\beta\big).
\end{aligned}\end{align}

For $f\in C^{\infty}(S)$, by (\ref{e.823}) and (\ref{e.824}) we have
\begin{align}\begin{aligned}\label{e.825}
&d\big((\pi^{*}f)(\alpha,\beta)\big)=d\big((\pi^{*}f)\alpha,(\pi^{*}_{\partial}f)\beta\big)\\
=&\Big((\pi^{*}d^{S}f)\wedge\alpha+(\pi^{*}f) d^{M}\alpha,i^{*}((\pi^{*}f)\alpha)-(\pi^{*}_{\partial}d^{S}f)\wedge\beta-(\pi^{*}_{\partial}f) d^{X}\beta\Big)\\
=&\pi^{*}(d^{S}f)\wedge(\alpha,\beta)+(\pi^{*}f)\cdot d(\alpha,\beta),
\end{aligned}\end{align}
so $d:\Omega(S,\Omega(Z,Y,F|_{Z}))\rightarrow \Omega(S,\Omega(Z,Y,F|_{Z}))$ verifies the Leibniz rule and defines a flat superconnection on $\Omega(Z,Y,F|_{Z})$.

As in \cite[$\S$ IIa)]{BL}, there is a $\mathbb{Z}-$graded vector bundle $H(Z, F)=\bigoplus_{p=0}^{m}H^{p}(Z, F)$ over $S$
whose fiber at $b\in S$ is the absolute cohomology group $H(Z_{b}, F)$. And there is another $\mathbb{Z}-$graded vector bundle $H(Z, Y, F)=\bigoplus_{p=0}^{m}H^{p}(Z, Y, F)$
 over $S$ whose fiber at $b\in S$ is the relative cohomology group $H(Z_{b}, Y_{b}, F)$.
Following the argument of \cite[$\S$III.(f)]{BL}, the flat superconnection $d$ defined in (\ref{e.825}) on $\Omega(Z,Y,F|_{Z})$ induces the canonical flat connection $\nabla^{H(Z, Y, F)}$ on $H(Z,Y, F)$.

The complex $\big(\Omega(Z,F|_{Z}),d^{Z}\big)$ computes the absolute de Rham cohomology groups along the fiber (cf. \cite[$\S$5.9]{Taylor96}).
Then $d^{M}:\Omega(S,\Omega(Z,F|_{Z}))\rightarrow \Omega(S,\Omega(Z,F|_{Z}))$ is a flat superconnection on $\Omega(Z,F|_{Z})$, it induces the canonical flat connection $\nabla^{H(Z,F)}$ on $H(Z,F)$.
By the Hodge theorem (cf. \cite[Thm. 1.1]{BruMa12}),
there are isomorphisms of smooth $\mathbb{Z}-$graded vector bundles on $S$
\begin{align}\begin{aligned}\label{e.590}
H(Z, F)\cong \Ker((D^{Z})^{2}_{\text{abs}})\quad \text{ and }\quad H(Z, Y, F)\cong \Ker((D^{Z})^{2}_{\text{rel}}).
\end{aligned}\end{align}
Let $h_{L^{2}}^{H(Z, F)}$ (resp. $h_{L^{2}}^{H(Z, Y, F)}$) be the $L^{2}-$metric on $H(Z, F)$ (resp. $H(Z, Y, F)$) induced by that of $\Ker((D^{Z})^{2}_{\text{abs}})$ (resp. $\Ker((D^{Z})^{2}_{\text{rel}})$) as a subbundle of $(\Omega(Z,F|_{Z}),h^{\Omega(Z,F|_{Z})})$ through the first (resp. second)
isomorphism in (\ref{e.590}).

\subsection{Double formula for heat kernel in family case}\label{ss1.5}

By the definition of fibration, we have an open covering $\mathscr{U}$ of $S$, such that for $U\in \mathscr{U}$, $\pi^{-1}(U)\cong U\times Z$. We work fiberwisely on $Z_{b}\,(b\in U)$. Recall that for $t>0$ the superconnection $C_{t}$ on $\Omega^{\bullet}(Z,F|_{Z})$ is defined in (\ref{e.549}). By (\ref{e.550}), (\ref{e.556}), we set
\begin{align}\begin{aligned}\label{e.659}
\mathscr{F}_{\rm bd}:=-4(D^{2}_{1})_{\rm bd}=4(C^{2}_{1})_{\rm bd}=(D^{Z})^{2}_{\rm bd}+\mathscr{F}^{[+]},
\end{aligned}\end{align}
where $(D^{Z})^{2}_{\rm bd}$ is the fiberwise Hodge Laplacian and $\mathscr{F}^{[+]}\in \Omega^{(>0)}(S, \text{End}(\Omega^{\bullet}(Z,F|_{Z})))$
 represents the higher degree
 part of $\mathscr{F}_{\rm bd}$.

\begin{defn}\label{d.39}(\textbf{Double fibration})
Let
$
\overline{Z}:=Z\cup_{Y} Z'
$\index{$\overline{Z}$}
be the double manifold, where $Z'$ is a copy of $Z$.
We denote $\overline{M}\overset{\pi}\longrightarrow S$ to be the double fibration of $M \overset{\pi}\longrightarrow S$, such that
\begin{align}\begin{aligned}\label{e.414}
\overline{M}=M\cup_{X} M' \quad {\rm with\, standard\, fiber} \quad \overline{Z},
\end{aligned}\end{align}
and $M'$ is a copy of $M$. Let $\phi:\overline{M}\rightarrow \overline{M}$ denote the
nature involution map, which keeps the boundary $X$ and exchanges $M$ and $M'$, such that $\phi^{2}=\Id$. Thus the group
\begin{align}\begin{aligned}\label{e.615}
\mathbb{Z}_{2}=\{\Id, \, \phi\}
\end{aligned}\end{align}
acts on $\overline{M}$. Using the product structures of $g^{TZ}$ and $h^{F}$ nearby the boundary $X$ (see (\ref{e.598}), (\ref{e.4}) and (\ref{e.6})), we construct $\mathbb{Z}_{2}-$invariant
 objects $g^{T\overline{Z}}$, $h^{\overline{F}}$ on $\overline{M}$ by gluing
$g^{TZ}$ and $h^{F}$ on $M$ and $M'$, i.e.,
\begin{align}\begin{aligned}\label{e.415}
g^{T\overline{Z}}=g^{TZ}\cup_{X} g^{TZ} \quad {\rm and} \quad h^{\overline{F}}=h^{F}\cup_{X} h^{F}.
\end{aligned}\end{align}
\end{defn}

We use a superscript ``$\overline{\quad}$" for the corresponding objects on $\overline{M}$. For any $t>0$, let $\exp(-t\overline{\mathscr{F}})(x,x')\,(x,x'\in Z_{b})$ be the smooth kernel of the heat operator $\exp(-t\overline{\mathscr{F}})$(along the fiber $\overline{Z}_{b}$), which is $C^{\infty}$
in $(t,x,x')\in (0,\infty)\times \overline{Z}_{b}\times \overline{Z}_{b}$.
 In what follows, we will not distinguish $\pi^{-1}(U)$ and $U\times \overline{Z}$. In particular, since $\overline{\mathscr{F}}_{b}$ is a smooth family of
second order elliptic differential operator, it was showed in \cite[Prop. 2.8]{Bismut86} and also \cite[Thm. 9.50]{BGV92} that $\exp(-t\overline{\mathscr{F}}_{b})(x,x')$ is $C^{\infty}$ in
$(t,x,x',b)\in (0,\infty)\times \overline{Z}\times \overline{Z}\times U$. For $x,x'\in \overline{Z}_{b}$, $\exp(-t\overline{\mathscr{F}}_{b})(x,x')$ is a linear mapping from
$\Lambda (T^{*}_{x}\overline{Z})\otimes F_{x}$ into $\Lambda (T^{*}_{b}S)\otimes \Lambda (T^{*}_{x'}\overline{Z})\otimes F_{x'}$.

\begin{defn}\label{d.35}
For convenience, we set
\begin{align}\begin{aligned}\label{e.613}
{\rm \mathbf{bd}}=\left\{
\begin{array}{ll}
 0, & \hbox{ for the absolute boundary conditions; }\\
 1, & \hbox{ for the relative boundary conditions.}
\end{array}
\right.\end{aligned}\end{align}
\end{defn}\index{${\rm \mathbf{bd}}$}

\begin{lemma}\label{l1.1} {\rm(\textbf{Double formula of heat kernels})}
For $e^{-t\mathscr{F}_{\rm bd}}$ with one of the natural boundary conditions {\rm (\ref{e.10}) or (\ref{e.11})}, the
heat kernel of $e^{-t\mathscr{F}_{\rm bd}}$ exists and is unique. And for $x,x'\in Z$ we have
\begin{align}\begin{aligned}\label{1.47}
e^{-t\mathscr{F}_{\rm bd}}(x, x')=e^{-t\overline{\mathscr{F}}}(x, x')+(-1)^{\rm\mathbf{bd}}\phi^{*}_{\phi^{-1}(x)}
\left( e^{-t\overline{\mathscr{F}}}(\phi^{-1}(x), x')\right).
\end{aligned}\end{align}
\end{lemma}
\begin{proof}(1) First, we show that for $t>0$ the right side of (\ref{1.47}) gives a heat kernel of $e^{-t\mathscr{F}_{\rm bd}}$. Since the geometric
 objects $g^{T\overline{Z}}$, $h^{\overline{F}}$, $\nabla^{\overline{F}}$ are $\mathbb{Z}_{2}-$invariant, the action of $\phi$ commutes with
$\overline{\mathscr{F}}$, i.e.,
\begin{align}\begin{aligned}\label{1.58}
 \overline{\mathscr{F}}\phi^{*}=\phi^{*}\overline{\mathscr{F}},
 \end{aligned}\end{align}
 so the right side of (\ref{1.47}) satisfies the heat equation.
 Next we verify that the right side of (\ref{1.47}) satisfies the initial condition when $t\rightarrow 0$, i.e.,
\begin{align}\begin{aligned}\label{1.59}
\lim_{t\rightarrow 0}\int_{Z}\left(e^{-t\overline{\mathscr{F}}}(x, x')+(-1)^{{\rm \textbf{bd}}}
\phi^{*}_{\phi^{-1}(x)}\big( e^{-t\overline{\mathscr{F}}}(\phi^{-1}(x), x')\big)\right)s(x')dv_{Z}(x')=s(x).
\end{aligned}\end{align}
For a section $s(x)\in \Lambda(T^{*}S)\otimes \Omega^{\bullet}_{{\rm bd }}(Z, F|_{Z})$, we define its double on $\overline{Z}$ to be
\begin{align}\begin{aligned}\label{1.60}
\overline{s}(x):=\left\{
\begin{array}{ll}
 s(x), & \quad x\in Z, \\
 (-1)^{{\rm \textbf{bd}}}\phi^{*}\Big(s\big(\phi^{-1}(x)\big)\Big), & \quad x\in Z'.
\end{array}
\right.
\end{aligned}\end{align}
Using the boundary conditions (\ref{e.10}) or (\ref{e.11}), we have
$\overline{s}(x)\in \Lambda(T^{*}S)\otimes C^{0}\big(\overline{Z}, \Lambda(T^{*}\overline{Z})\otimes \overline{F}\big)$ and
\begin{align}\begin{aligned}\label{e.764}
\left(\phi^{*}\overline{s}\right)(x)=(-1)^{{\rm \textbf{bd}}}\overline{s}(x),\quad \text{for all}\, x\in \overline{Z}.
\end{aligned}\end{align}

Then for $x\in Z$, we have
\begin{align}\begin{aligned}\label{1.61}
_{}\lim_{t\rightarrow 0}\int_{\overline{Z}}e^{-t\overline{\mathscr{F}}}(x, x')\overline{s}(x')dv_{\overline{Z}}(x')=\overline{s}(x)=s(x).
\end{aligned}\end{align}
By (\ref{1.60}), we have
\begin{align}\begin{aligned}\label{1.62}
&\int_{\overline{Z}}e^{-t\overline{\mathscr{F}}}(x, x')\overline{s}(x')dv_{\overline{Z}}(x')\\
&=\int_{Z}e^{-t\overline{\mathscr{F}}}(x,x') s(x')dv_{Z}(x')+(-1)^{{\rm \textbf{bd}}}\int_{Z'}e^{-t\overline{\mathscr{F}}}(x,x') \phi^{*}\Big(s\big(\phi^{-1}(x')\big)\Big)dv_{Z'}(x').
\end{aligned}\end{align}

By (\ref{1.58}), we get $\exp(-t\overline{\mathscr{F}})=\phi^{*}\exp(-t\overline{\mathscr{F}})\left(\phi^{*}\right)^{-1}$ as operators, thus
we have
\begin{align}\begin{aligned}\label{e.765}
e^{-t\overline{\mathscr{F}}}(x,x')=\phi^{*}_{\phi^{-1}(x)}e^{-t\overline{\mathscr{F}}}(\phi^{-1}(x),\phi^{-1}(x'))\left(\phi^{*}\right)^{-1}_{x'}.
\end{aligned}\end{align}
In fact, we have
$$\begin{aligned}
\big(e^{-t\overline{\mathscr{F}}}s\big)(x)&=\phi^{*}\big(e^{-t\overline{\mathscr{F}}}((\phi^{*})^{-1}\cdot s)\big)(x)
=\int_{\overline{Z}}\phi^{*}_{x}\cdot e^{-t\overline{\mathscr{F}}}(\phi^{-1}(x),x')(\phi^{*})^{-1}_{x'}s(\phi(x'))dv_{\overline{Z}}(x')\\
&=\int_{\overline{Z}}\phi^{*}_{x}\cdot e^{-t\overline{\mathscr{F}}}(\phi^{-1}(x),\phi^{-1}(x'))(\phi^{*})^{-1}_{x'}s(x')dv_{\overline{Z}}(x').
\end{aligned}$$
By (\ref{e.765}) we get
\begin{align}\begin{aligned}\label{1.65}
\int_{Z'}&e^{-t\overline{\mathscr{F}}}(x, x') \phi^{*}\Big(s\big(\phi^{-1}(x)\big)\Big)dv_{Z'}(x')\\
&=\int_{Z'}\phi^{*}_{\phi^{-1}(x)}e^{-t\overline{\mathscr{F}}}(\phi^{-1}(x), \phi^{-1}(x')) s(\phi^{-1}(x'))
dv_{Z'}(x')\\
&=\int_{Z}\phi^{*}_{\phi^{-1}(x)}e^{-t\overline{\mathscr{F}}}(\phi^{-1}(x), z) s(z)dv_{Z}(z), \quad x'=\phi(z).
\end{aligned}\end{align}
From (\ref{1.61}), (\ref{1.62}) and (\ref{1.65}), we get (\ref{1.59}).\\

Next we try to show that the right side of (\ref{1.47}) verifies the natural boundary conditions (\ref{e.10}) or (\ref{e.11}).
 Let $Y_{[-\varepsilon, \varepsilon]}$ be the product neighborhood of $Y\subset \overline{Z}$ .
For $(y, x_{m})\in Y_{[-\varepsilon, \varepsilon]}$, we have
\begin{align}\begin{aligned}\label{1.66}
 (y, x_{m})\overset{\phi}\longrightarrow (y, -x_{m})\quad \text{and} \quad \phi^{*}dx_{m}=-dx_{m}.
\end{aligned}\end{align}
For $x\in Y_{[-\varepsilon, \varepsilon]}$, $e^{-t\overline{\mathscr{F}}}(x, x')\in \Lambda(T^{*}S)\otimes \big(\Lambda (T^{*}\overline{Z})\otimes \overline{F}\big)_{x}\otimes\big(\Lambda (T^{*}\overline{Z})\otimes \overline{F}\big)^{*}_{x'}$ can be expressed as
\begin{align}\begin{aligned}\label{1.67}
 e^{-t\overline{\mathscr{F}}}(x, x')=dx_{m}\wedge f\big((y, x_{m}),x'\big)+ g\big((y, x_{m}),x'\big), \quad x=(y, x_{m}),
\end{aligned}\end{align}
such that
\begin{align}\begin{aligned}\label{1.68}
(\phi^{*}_{x} f)\big((y, x_{m}),x'\big)&= f\big((y, -x_{m}),x'\big), \\
(\phi^{*}_{x} g)\big((y, x_{m}),x'\big)&= g\big((y, -x_{m}),x'\big).
\end{aligned}\end{align}
For (\ref{1.66}), (\ref{1.67}) and (\ref{1.68}), we get that for $x=(y, x_{m})\in Y_{[-\varepsilon, \varepsilon]}$
\begin{align}\begin{aligned}\label{1.69}
\phi^{*}_{x}e^{-t\overline{\mathscr{F}}}(\phi^{-1}(x), x')=-dx_{m}\wedge f\big((y, -x_{m}),x'\big)+ g\big((y, -x_{m}),x'\big).
\end{aligned}\end{align}
From (\ref{1.67}) and (\ref{1.69}), we see that for the absolute boundary conditions
\begin{align}\begin{aligned}\label{1.70}
i_{\frac{\partial}{\partial x_{m}}}&\left(e^{-t\overline{\mathscr{F}}}(x, x')+\phi^{*}_{x}e^{-t\overline{\mathscr{F}}}(\phi^{-1}(x), x')\right)_{x_{m}=0}\\
&=\Big(f((y, x_{m}),x')-f((y, -x_{m}),x')\Big)_{x_{m}=0}=0.
\end{aligned}\end{align}
Now we have verified the first equality of (\ref{e.10}). For the second we have on $Y_{[-\varepsilon, \varepsilon]}$
\begin{align}\begin{aligned}\label{1.71}
d^{Z}=dx_{m}\wedge\frac{\partial}{\partial x_{m}}+d^{Y},
\end{aligned}\end{align}
then it follows from (\ref{1.67}), (\ref{1.69}) and (\ref{1.71}) that
\begin{align}\begin{aligned}\label{1.72}
i_{\frac{\partial}{\partial x_{m}}}d^{Z}&\left(e^{-t\overline{\mathscr{F}}}(x, x')+\phi^{*}_{x}e^{-t\overline{\mathscr{F}}}(\phi^{-1}(x), x')\right)_{x_{m}=0}\\
&=\Big(-(d^{Y}f)\big((y, x_{m}),x'\big)+(d^{Y}f)\big((y, -x_{m}),x'\big)\\
&\quad\quad\quad\quad\quad\quad\quad\quad+\frac{\partial g}{\partial x_{m}}\big((y, x_{m}),x'\big)-\frac{\partial g}{\partial x_{m}}\big((y, -x_{m}),x'\big)\Big)_{x_{m}=0}\\
&=0.
\end{aligned}\end{align}
In a similar way, we trait the case of relative boundary conditions. Now we have shown that the right side of (\ref{1.47})
is a heat kernel of $e^{-t\mathscr{F}_{\rm bd}}$.
The double formula (\ref{1.47}) is a natural consequence of the uniqueness of the heat kernel of $e^{-t\mathscr{F}_{\rm bd}}$.
\end{proof}

For $t>0$, we define the rescaling operator $\psi_{t}\in \text{End}(\Omega(S))$ such that for $\alpha \in \Omega^{k}(S)$
\begin{align}\begin{aligned}\label{e.731}
\psi_{t}\alpha=t^{-k/2}\alpha.
\end{aligned}\end{align}
Then by \cite[Prop. 3.17]{BGo} and (\ref{e.550}), we have
\begin{align}\begin{aligned}\label{e.732}
D_{t}=\sqrt{t}\psi^{-1}_{t}D_{1}\psi_{t}.
\end{aligned}\end{align}

\begin{lemma}\label{l.46} We have the double formula of the kernel of $4D^{2}_{1}e^{4tD^{2}_{1}}$ with the natural boundary conditions (\ref{e.10}) or (\ref{e.11}),
\begin{align}\begin{aligned}\label{1.74}
f'(D_{t})(x, x')=f'(\overline{D}_{t})( x, x')
+(-1)^{{\rm \mathbf{bd}}}\phi^{*}_{\phi^{-1}(x)}\left(f'(\overline{D}_{t})\right)(\phi^{-1}(x), x').
\end{aligned}\end{align}
\end{lemma}
\begin{proof} Since $\phi^{*}$ is isometric, we have $\phi^{*}D^{2}_{1}=D^{2}_{1}\phi^{*}$. By (\ref{e.659}) and Lemma \ref{l1.1}, we have
\begin{align}\begin{aligned}\label{e.763}
\left(4D^{2}_{1}e^{4tD^{2}_{1}}\right)(x, x')
=\left(4\overline{D}^{2}_{1}e^{4t\overline{D}^{2}_{1}}\right)&( x, x')\\
&+(-1)^{{\rm \mathbf{bd}}}\phi^{*}_{\phi^{-1}(x)}\left(4\overline{D}^{2}_{1}e^{4t\overline{D}^{2}_{1}}\right)(\phi^{-1}(x), x').
\end{aligned}\end{align}
Then (\ref{1.74}) follows from (\ref{e.732}), (\ref{e.763}) and $f'(x)=(1+2x^{2})e^{x^{2}}$.
\end{proof}

\subsection{Analytic torsion form in the case with boundary}\label{ss1.7}

Recall that $m=\dim Z$. Let $\text{Pf}:so(m)\rightarrow \mathbb{R}$ denote the Pfaffian. Let $R^{TZ}$ be the curvature of $\nabla^{TZ}$ defined in (\ref{e.551}). Set
\begin{align}\begin{aligned}\label{e.562}
e(TZ, \nabla^{TZ})=\text{Pf}\Big[\frac{R^{TZ}}{2\pi}\Big].
\end{aligned}\end{align}
Then $e(TZ, \nabla^{TZ})$ is an $o(TZ)-$value closed m-form on $M$. In the case of boundary, following Br\"{u}ning-Ma \cite[(1.43)]{BruMa06} we define the relative Euler form of $TZ$ associated to $\nabla^{TZ}$, $E(TZ,\nabla^{TZ})=\big(e(TZ,\nabla^{TZ}),e_{b}(X,\nabla^{TZ})\big)\in \big(\Omega^{m}(M,o(TZ)),\Omega^{m-1}(X,o(TZ))\big)$, then we have
\begin{enumerate}
  \item If $m$ is even and (\ref{e.4}) holds, then $e_{b}(X,\nabla^{TZ})=0$.
  \item If $m$ is odd, then $e(TZ,\nabla^{TZ})=0$ and $e_{b}(X,\nabla^{TZ})=\frac{1}{2}e(TY, \nabla^{TY})$.
\end{enumerate}

Put
\begin{align}\begin{aligned}\label{e.563}
\chi(Z)=\sum_{i=0}^{m}(-1)^{i}\rk (H^{i}(Z, \mathbb{C})), \quad \chi'(Z, F)=\sum_{i=0}^{m}(-1)^{i}i\rk( H^{i}(Z, F)).
\end{aligned}\end{align}
Then $\chi(Z)$ is the Euler characteristic of $Z$.
We define a notation $H_{{\rm bd }}(Z,F)$ such that
\begin{align}\begin{aligned}\label{e.616}
H_{{\rm bd }}(Z,F)=\left\{
\begin{array}{ll}
 H(Z, F) & \hbox{ for the absolute boundary conditions;} \\
 H(Z, Y, F) & \hbox{ for the relative boundary conditions.}
\end{array}
\right.\end{aligned}\end{align}
 For $f(a)$ in (\ref{e.762}), as in Definition \ref{d.23}, we put
\begin{align}\begin{aligned}\label{e.567}
f(\nabla^{F}, h^{F})&=(2i\pi)^{1/2}\varphi \tr\Big[f\left(\frac{\omega}{2}(F, h^{F})\right)\Big]\in \Omega (M), \\
f(\nabla^{H_{{\rm bd }}(Z, F)}, h_{L^{2}}^{H_{{\rm bd }}(Z, F)})&=(2i\pi)^{1/2}\varphi \tr_{s}\Big[f\left(\frac{\omega}{2}\big(H_{{\rm bd }}(Z, F), h_{L^{2}}^{H_{{\rm bd }}(Z, F)}\big)\right)\Big]\in \Omega (S).
\end{aligned}\end{align}
For any $ t>0$, the operator $D_{t}$ in (\ref{e.549}) is a first-order fiberwise-elliptic differential operator, then $f(D_{t})$ is a fiberwise trace class operator. For $t>0$, we put:
\begin{align}\begin{aligned}\label{e.568}
&f(C'_{t}, h^{\Omega^{\bullet}(Z,F|_{Z})})=(2i\pi)^{1/2}\varphi \tr_{s}[f(D_{t})]\in \Omega (S), \\
&f^{\wedge}(C'_{t}, h^{\Omega^{\bullet}(Z,F|_{Z})}):=\varphi \tr_{s}\left[\frac{N}{2}f'(D_{t})\right]=\varphi \tr_{s}\left[\frac{N}{2}(1+2D_{t}^{2})e^{D_{t}^{2}}\right].\end{aligned}\end{align}

\begin{thm}\label{t1.2} For any $t>0$, the form $f^{\wedge}(C'_{t}, h^{\Omega^{\bullet}(Z,F|_{Z})})$ is real and even. Moreover,
\begin{align}\begin{aligned}\label{1.84}
\frac{\partial}{\partial t}f(C'_{t}, h^{\Omega^{\bullet}(Z,F|_{Z})})=\frac{1}{t}df^{\wedge}(C'_{t}, h^{\Omega^{\bullet}(Z,F|_{Z})}).
\end{aligned}\end{align}
\end{thm}
\begin{proof} Proceeding as in the proof of \cite[Thm. 2.11]{BL}, we get (\ref{1.84}) by using the deformation argument in \cite[Thm. 1.9]{BL}.
\end{proof}

For the fibration of the boundary $\pi_{\partial}: X \rightarrow S$ with the objects $T^{H}X,g^{TY}, \nabla^{F|_{X}}, h^{F|_{X}}$ induced by
 $T^{H}M,g^{TZ}, \nabla^{F}, h^{F}$, we define $\widetilde{C}_{t}$, $\widetilde{D}_{t}$ ($t>0$) as in (\ref{e.549}). We will use the superscript ``$\widetilde{\quad}$" for the corresponding objects on $X$.

To define the Bismut-Lott torsion form in the case with boundary,
the main difference is to establish the following theorem corresponding to \cite[Thm. 3.16]{BL} in the case without
boundary.

\begin{thm}\label{t1.1}
For $\pi:M\rightarrow S$ (see Figure \ref{fig.1}) and any $t>0$, the form $f(C'_{t}, h^{\Omega^{\bullet}(Z,F|_{Z})})$ is real, odd and closed. Its de Rham cohomology class is independent of the choices of $t, \, T^{H}M, \, g^{TZ}$ and $h^{F}$, which have the product structures on the same neighborhood $X_{[-\varepsilon,0]}$. For the natural boundary conditions (\ref{e.10}) or (\ref{e.11}), we have
\begin{enumerate}
 \item As $t\rightarrow 0$,
\begin{align}\begin{aligned}\label{1.93}
 f(C'_{t}, h^{\Omega^{\bullet}(Z, F|_{Z})})=\int_{Z}&e(TZ, \nabla^{TZ})f(\nabla^{F}, h^{F})\\
&+ (-1)^{{\rm \mathbf{bd}}}\frac{1}{2}\int_{Y}e(TY, \nabla^{TY})f(\nabla^{F}, h^{F}) +\mathcal{O}(\sqrt{t}).
\end{aligned}\end{align}
 \item As $t\rightarrow +\infty$,
 \begin{align}\label{1.94}
   f(C'_{t}, h^{\Omega^{\bullet}(Z, F|_{Z})})=f(\nabla^{H_{{\rm bd }}(Z,F)}, h^{H_{{\rm bd }}(Z,F)})+\mathcal{O}(\frac{1}{\sqrt{t}}).
\end{align}
We remark that if $m$ is even {\rm (resp.  odd)}, then
 \begin{align}\label{e.766}
 e(TY, \nabla^{TY})=0,\quad (\text{\rm resp.}\,  e(TZ, \nabla^{TZ})=0).
\end{align}
\end{enumerate}
\end{thm}
\begin{proof} Proceeding as in the proof \cite[Prop. 1.3, Thm. 1.8]{BL}, it follows that $f(C'_{t}, h^{\Omega^{\bullet}(Z, F|_{Z})})$ is real, odd and closed. By Theorem \ref{t1.2}, we see that the de Rham cohomology class of $f(C'_{t}, h^{\Omega^{\bullet}(Z, F|_{Z})})$ is independent of $t$. Let $(T'^{H}M, g'^{TZ}, h'^{F})$ be another triple of arguments verifying the assumption (\ref{e.761}) of product structures on the same product neighborhood $X_{[-\varepsilon,0]}$ as $(T^{H}M, g^{TZ}, h^{F})$. We connect the two triples linearly by a smooth path $\gamma(t):=(T^{H}_{t}M, g^{TZ}_{t}, h^{F}_{t})$($t\in [0,1]$) such that $\gamma(0)=(T^{H}M, g^{TZ}, h^{F})$, $\gamma(1)=(T'^{H}M, g'^{TZ}, h'^{F})$, then $(T^{H}_{t}M, g^{TZ}_{t}, h^{F}_{t})$ verify the assumption (\ref{e.761}) on $X_{[-\varepsilon,0]}$ for all $t\in [0,1]$. Following the proof of \cite[Thm. 1.9]{BL}, we see that the de Rham cohomology class is independent of $T^{H}M, g^{TZ}, h^{F}$.

For the infinite cylinder $X_{\mathbb{R}}$, let $\psi:X_{\mathbb{R}}\rightarrow X$ be the projection map and
$F_{c}:=\psi^{*}(F|_{X})$ be the flat vector bundle on $X_{\mathbb{R}}$
with a flat connection $\nabla^{F_{c}}:=\psi ^{*}(\nabla ^{F}|_{X})$. Now we extend the operator $D_{t}$ in (\ref{e.549}) from $X_{[-\varepsilon,\varepsilon]}\subset M$ to $X_{\mathbb{R}}$ by
 \begin{align}\label{e.767}
 D_{c, t}:=\widetilde{D}_{t}-\frac{\sqrt{t}}{2}\widehat{c}(dx_{m})\frac{\partial}{\partial x_{m}},
\end{align}
where $\widetilde{D}_{t}$ is the operator associated to $(X, F)$.
 It follows from Lemma \ref{l.46} that
\begin{align}\begin{aligned}\label{e.476}
&f(C'_{t}, h^{\Omega^{\bullet}(Z, F|_{Z})})=(2i\pi)^{1/2}\varphi \int_{Z}\tr_{s}\left[f(D_{t})(x, x)\right]dv_{Z}\\
&= (2i\pi)^{1/2}\varphi \int_{Z}\tr_{s}\left[f(\overline{D}_{t})(x, x)\right]dv_{Z}\\
&\quad \quad \quad \qquad+ (-1)^{{\rm \mathbf{bd}}}(2i\pi)^{1/2}\varphi \int_{Z}\tr_{s}\left[\left(\phi^{*}_{\phi^{-1}(x)}f(\overline{D}_{t})(\phi^{-1}(x), x')\right)\big|_{x=x'}\right]dv_{Z}\\
&:=I_{1}(t)+(-1)^{{\rm \mathbf{bd}}}I_{2}(t).\end{aligned}\end{align}
For the first term $I_{1}$, by \cite[Thm. 3.16]{BL} we have
\begin{align}\begin{aligned}\label{e.477}
I_{1}(t)=\left\{\begin{array}{ll}
         \int_{Z}e(TZ, \nabla^{TZ})f(\nabla^{F}, h^{F})+\mathcal{O}(t)  & \hbox{if $m$ is even;} \\
         \mathcal{O}(\sqrt{t}) & \hbox{if $m$ is odd.}
        \end{array}
\right.\end{aligned}\end{align}
For $\varepsilon>0$, let $Z_{\frac{\varepsilon}{2}}:=Z\backslash Y _{[-\frac{\varepsilon}{2},0]}$. We set
\begin{align}\begin{aligned}\label{1.50}
&I_{3}(t)=(2i\pi)^{1/2}\varphi\int_{Z_{\frac{\varepsilon}{2}}}\tr_{s}\Big[\left(\phi^{*}_{\phi^{-1}(x)}f(\overline{D}_{t})(\phi^{-1}(x), x')\right)\big|_{x=x'}\Big]dv_{Z}, \\
&I_{4}(t)=(2i\pi)^{1/2}\varphi\int_{Y _{ [-\frac{\varepsilon}{2},0]}}\tr_{s}\Big[\left(\phi^{*}_{\phi^{-1}(x)}f(\overline{D}_{t})(\phi^{-1}(x), x')\right)\big|_{x=x'}\Big]dv_{Z}.
\end{aligned}\end{align}
Then we have
\begin{align}\begin{aligned}\label{e.618}
I_{2}(t)=I_{3}(t)+I_{4}(t).
\end{aligned}\end{align}
Since ${\rm d}(x, \phi^{-1}(x))\geq \varepsilon $ for $x \in Z_{\frac{\varepsilon}{2}}$, by applying the off-diagonal estimate of the heat kernel,
for $T>0$ there exist constants $C, c>0$, such that for any $0<t<T$ and $ x \in Z_{\frac{\varepsilon}{2}}$
\begin{align}\begin{aligned}\label{e.478}
\Big|\left(\overline{D}_{t}\exp(\overline{D}^{2}_{t})\right)(\phi^{-1}(x), x)\Big|_{\mathscr{C}^{0}}\leq Ce^{-c\frac{\varepsilon^{2}}{t}}.
\end{aligned}\end{align}
It follows that $I_{3}(t)=\mathcal{O}(e^{-c/t})$, then by (\ref{e.476}), (\ref{e.477}) and (\ref{e.478}) we get
\begin{align}\begin{aligned}\label{1.77}
f(C'_{t}, h^{\Omega^{\bullet}(Z, F|_{Z})})=
\left\{\begin{array}{ll}
         \int_{Z}e(TZ, \nabla^{TZ})f(\nabla^{F}, h^{F})+(-1)^{{\rm \mathbf{bd}}}I_{4}(t)+\mathcal{O}(t)  & \hbox{if $m$ is even;} \\
         (-1)^{{\rm \mathbf{bd}}}I_{4}(t)+\mathcal{O}(\sqrt{t}) & \hbox{if $m$ is odd.}
        \end{array}
\right.\end{aligned}\end{align}
By using the finite propagation speed property of the wave equation (cf. \cite[Appendix D.2]{MaMa07}), we compare $f(\overline{D}_{t})( \phi^{-1}(x), x)$ restricted on
$Y_{[-\frac{\varepsilon}{2}, \frac{\varepsilon}{2}]}\subset \overline{Z}$ and
 $f(D_{c, t})(\phi^{-1}(x), x)$ restricted on $Y_{[-\frac{\varepsilon}{2}, \frac{\varepsilon}{2}]} \subset Y_{ \mathbb{R}}$, since we have $\overline{D}^{2}_{t}|_{Y_{[-\varepsilon,\varepsilon]}}=D^{2}_{c, t}|_{Y_{[-\varepsilon,\varepsilon]}}=\widetilde{D}^{2}_{t}-\frac{t}{4}\frac{\partial^{2}}{\partial x^{2}_{m}}$.
It follows that for any $0<t<T$ and $x \in Y_{[-\frac{\varepsilon}{2}, \frac{\varepsilon}{2}]}$
\begin{align}\begin{aligned}\label{1.78}
\big|f(\overline{D}_{t})(\phi^{-1}(x), x)-f(D_{c, t})(\phi^{-1}(x), x )\big|_{\mathscr{C}^{0}}\leq Ce^{-c\frac{\varepsilon^{2}}{4t}},
\end{aligned}\end{align}
hence by (\ref{1.50}) and (\ref{1.78}) we have
\begin{align}\begin{aligned}\label{1.79}
I_{4}(t)=(2i\pi)^{1/2}\varphi \int_{-\frac{\varepsilon}{2}}^{0}\int_{Y}\tr_{s}\left[\left(\phi^{*}_{\phi^{-1}(x)}f(D_{c, t})(\phi^{-1}(x), x')
\right)\big|_{x=y}\right]&dv_{Y}dx_{m}\\&+\mathcal{O}(e^{-c/t}).
\end{aligned}\end{align}
Let $\{\xi_{I}, \xi_{I}\wedge e^{\bf{n}} \}$ be an orthonormal frame of $\Lambda(N^{*})\otimes\Lambda (T^{*}Y)\otimes F$ on $X _{[-\frac{\varepsilon}{2},0]}$, then
\begin{align}\begin{aligned}\label{e.768}
\phi^{*}(\xi_{I})=\xi_{I},\quad\phi^{*}(e^{\bf{n}})=-e^{\bf{n}}\, \text{and}\quad \phi[(y, x_{m})]=(y, -x_{m})\in Y _{[-\frac{\varepsilon}{2}, 0]}.
 \end{aligned}\end{align}
Then by (\ref{e.767}), (\ref{e.768}) we have
\begin{align}\begin{aligned}\label{1.80}
&\tr_{s}\left[\left(\phi^{*}_{\phi^{-1}(x)}f(D_{c, t})(\phi^{-1}(x), x')\right)\big|_{x=x'}\right]\\
&=\sum (-1)^{|I|}\left(\Big((\widetilde{D}_{t}-\frac{\sqrt {t}}{2} \widehat{c}(dx_{m}))e^{\widetilde{D}^{2}_{t}-\frac{t}{4}\frac{\partial ^{2}}{\partial x^{2}_{m}}}\Big)(\phi^{-1}(x), x )\xi_{I}, \phi^{*}_{x}(\xi_{I})\right)\\
  &+\sum (-1)^{|I|+1}\left(\Big((\widetilde{D}_{t}-\frac{\sqrt {t}}{2} \widehat{c}(dx_{m}))e^{\widetilde{D}^{2}_{t}-\frac{t}{4}\frac{\partial ^{2}}{\partial x^{2}_{m}}}\Big)(\phi^{-1}(x), x )\xi_{I}\wedge e^{\bf{n}}, \phi^{*}_{x}(\xi_{I}\wedge e^{\bf{n}})\right)\\
  &=\sum (-1)^{|I|}\left(\Big(\widetilde{D}_{t}e^{\widetilde{D}^{2}_{t}-\frac{t}{4}\frac{\partial ^{2}}{\partial x^{2}_{m}}}\Big)(\phi^{-1}(x), x )\xi_{I}, \xi_{I}\right)\\
  &+\sum (-1)^{|I|+1}\left(\Big(\widetilde{D}_{t}e^{\widetilde{D}^{2}_{t}-\frac{t}{4}\frac{\partial ^{2}}{\partial x^{2}_{m}}}\Big)(\phi^{-1}(x), x )\xi_{I}\wedge e^{\bf{n}}, -\xi_{I}\wedge e^{\bf{n}}\right)\\
&=2e^{-\frac{t}{4}\frac{\partial ^{2}}{\partial x^{2}_{m}}}(-x_{m}, x_{m})\cdot \tr_{s}\left[\left(\widetilde{D}_{t}e^{\widetilde{D}_{t}^{2}}\right)(y, y)\right].
\end{aligned}\end{align}
To get the second equality of (\ref{1.80}), we have used the fact that $\widehat{c}(dx_{m})$ exchanges the $\mathbb{Z}_{2}-$graduation
of $\Lambda(N^{*})$.
By \cite[Thm. 3.16]{BL} and (\ref{1.80}), we get as $t\rightarrow 0$
\begin{align}\begin{aligned}\label{1.81}
I_{4}(t)
&=2 \int_{-\frac{\varepsilon}{2}}^{0}\frac{e^{-\frac{4x^{2}_{m}}{t}}}{\sqrt{\pi t}}dx_{m}\cdot (2i\pi)^{1/2}\varphi\int_{Y}\tr_{s}\left[\left(\widetilde{D}_{t}e^{\widetilde{D}_{t}^{2}}\right)(y, y)\right]dv_{Y}(y)+\mathcal{O}(e^{-c/t})\\
&=\left\{
\begin{array}{ll}
\mathcal{O}(\sqrt{t})  & \hbox{if $m$ is even;}\\
\frac{1}{2}\int_{Y}e(TY, \nabla^{TY})f(\nabla^{F}, h^{F})+\mathcal{O}(t) & \hbox{if $m$ is odd.}
\end{array}
\right.\end{aligned}\end{align}
Finally from (\ref{1.77}) and (\ref{1.81}), we get (\ref{1.93}). The proof of (\ref{1.94}) is exactly the same as that of \cite[(3.85)]{BL}.
\end{proof}

Put
\begin{align}\begin{aligned}\label{e.37}
&\chi'(Z, F)=\sum ^{m}_{p=0}(-1)^{p}p \cdot \rk(H^{p}(Z, F)),\\
& \chi'(Z,Y, F)=\sum ^{m}_{p=0}(-1)^{p}p \cdot \rk(H^{p}(Z, Y, F)).
\end{aligned}\end{align}\index{$\chi'_{{\rm bd }}(Z, F),\,\chi'(Z, F),\,\chi'(Z,Y, F)$}

We define some notations such that
\begin{align}\begin{aligned}\label{e.621}
\chi_{{\rm bd }}(Z)=\left\{
\begin{array}{ll}
 \chi(Z,\mathbb{C})  & \hbox{ for absolute boundary conditions; } \\
 \chi(Z, Y,\mathbb{C})  & \hbox{ for relative boundary conditions.}
\end{array}
\right.\\
\chi'_{{\rm bd }}(Z, F)=\left\{
\begin{array}{ll}
 \chi'(Z, F)  & \hbox{for absolute boundary conditions; } \\
 \chi'(Z, Y, F)  & \hbox{for relative boundary conditions.}
\end{array}
\right.
\end{aligned}\end{align}

\begin{thm}\label{t1.4} For the natural boundary conditions {\rm (\ref{e.10}) or (\ref{e.11})} as $t\rightarrow 0$,
\begin{align}\label{1.95}
 f^{\wedge}(C'_{t}, h^{\Omega^{\bullet}(Z,F|_{Z})})=\frac{1}{4}m\chi_{{\rm bd}}(Z)\rk(F) +\mathcal{O}(\sqrt{t}),
\end{align}
 as $t\rightarrow +\infty$,
 \begin{align}\label{1.96}
 f^{\wedge}(C'_{t}, h^{\Omega^{\bullet}(Z,F|_{Z})})=\frac{\chi'_{\rm bd}(Z, F)}{2}+\mathcal{O}(\frac{1}{\sqrt{t}}).
\end{align}
\end{thm}
\begin{proof}
By using Theorem \ref{t1.1}, the proof is essentially the same as \cite[Thm. 3.21]{BL}.
 \end{proof}

\begin{defn}\label{d.6}{\bf (Analytic Torsion Form)}\\ The analytic torsion form
$\mathscr{T}_{{\rm bd }}(T^{H}M, g^{TZ}, h^{F})\in \Omega(S)$ with the natural boundary conditions {\rm (\ref{e.10}) or (\ref{e.11})} is defined by
\begin{align}\begin{aligned}\label{e.43}
\mathscr{T}_{{\rm bd }}&(T^{H}M, g^{TZ}, h^{F})=-\int_{0}^{+\infty}\left[f^{\wedge}(C'_{t}, h^{\Omega^{\bullet}(Z,F|_{Z})})-\frac{\chi'_{{\rm bd }}(Z, F)}{2}f'(0)\right.\\
&\qquad \quad\qquad\qquad  \quad\left.-\Big(\frac{1}{4}m\rk(F)\chi_{{\rm bd }}(Z)-\frac{\chi'_{{\rm bd }}(Z, F)}{2}\Big)f'\Big(\frac{i\sqrt{t}}{2}\Big)\right]\frac{dt}{t}.\end{aligned}\end{align}
\end{defn}\index{$\mathscr{T}_{{\rm bd }}(T^{H}M, g^{TZ}, h^{F})$}

It follows from Theorem \ref{t1.4} that the integrand of (\ref{e.43}) is integrable on $[0, \infty]$. The above definition extends the analytic torsion form of Bismut and Lott (cf. \cite[Def. 3.22]{BL}) to the boundary case.
The next theorem is the differential form version of Bismut-Lott's ``Riemann-Roch-Grothendieck" theorem in the case with boundary.
\begin{thm}\label{t.3} The torsion form $\mathscr{T}_{{\rm bd }}(T^{H}M, g^{TZ}, h^{F})$ is even and real. Moreover, we have
\begin{align}\begin{aligned}\label{e.49}
d \mathscr{T}_{{\rm bd }}(T^{H}M, &g^{TZ}, h^{F})
=\int_{Z}e(TZ, \nabla^{TZ})f(\nabla^{F}, h^{F})\\
&+(-1)^{{\rm \mathbf{bd}}}\frac{1}{2}\int_{Y}e(TY, \nabla^{TY})f(\nabla^{F}, h^{F})-f(\nabla^{H_{{\rm bd }}(Z,F)}, h^{H_{{\rm bd }}(Z,F)}).
\end{aligned}\end{align}
\end{thm}
\begin{proof}
This theorem is a direct consequence of Theorems \ref{t1.1}, \ref{t1.2} and Definition \ref{d.6}.
\end{proof}

\section{ Fiberwise well-defined Morse function and spectral sequences}\label{s.2}
 Let  $h:M\rightarrow \mathbb{R}$ be a fiberwise Morse function, such that $h|_{X}$ is also a fiberwise Morse
function and verifies the Smale transversality conditions. Let $\big(C^{\bullet}(W^{u}_{Z}, F), \widetilde{\partial}, \nabla^{C^{\bullet}(W^{u}_{Z}, F)}\big)$
be the flat vector bundles on $S$ of the Thom-Smale complex
 associated to $h$.

Let $\big(C^{\bullet}(W^{u}_{Z_{1}}, F), \widetilde{\partial}, \nabla^{C^{\bullet}(W^{u}_{Z_{1}}, F)}\big)$ (resp. $\big(C^{\bullet}(W^{u}_{Z_{2}}/W^{u}_{Y}, F), \widetilde{\partial}, \nabla^{C^{\bullet}(W^{u}_{Z_{2}}/W^{u}_{Y}, F)}\big)$) be the flat vector bundles on $S$ of the absolute (resp. relative) Thom-Smale co-chain complex associated to $h|_{M_{1}}$ (resp.  $h|_{M_{2}}$ ).
We associate three combinatorial torsion forms to the above complexes of flat vector bundles:
$T_{f}(A^{C^{\bullet}(W^{u}_{Z}, F)}, h^{C^{\bullet}(W^{u}_{Z}, F)}),$ $T_{f}(A^{C^{\bullet}(W^{u}_{Z_{1}}, F)}, h^{C^{\bullet}(W^{u}_{Z_{1}}, F)}),$ $
 T_{f}(A^{C^{\bullet}(W^{u}_{Z_{2}}/W^{u}_{Y}, F)}, h^{C^{\bullet}(W^{u}_{Z_{2}}/W^{u}_{Y}, F)}).$ These three complexes of flat vector bundles form a double complex
 of flat vector bundles (see (\ref{e.373})).

This section is organized as follows. In Section \ref{ss2.1}, we explain the related concepts of Morse theory when there exists a fiberwise
Morse function.
In Section \ref{ss2.2}, We assign each spectral sequence $\big(E_{k}, d_{k}\big)$ of the double complex
mentioned above a combinatorial torsion form $T_{f}\big(A^{E_{k}}_{k}, h^{E_{k}}\big)$.
Then by a result of Goette we prove an equation of these combinatorial torsion forms.
This equation implies a ``gluing" relation among the three combinatorial torsion forms mentioned in the preceding paragraph. In Section \ref{ss2.4},
we introduce the double fibration and establish the double formula of analytic torsion form.
In Section \ref{ss2.5}, we establish the double formula of combinatorial torsion form.

\subsection{Thom-Smale complex of a fiberwise gradient vector field}\label{ss2.1}
Recall that $\pi:M\rightarrow S$ is a fibration divided into two fibrations $M_{1}$
and $M_{2}$ by a hypersurface $X$ (see Figure \ref{fig.5}). The reader can refer to \cite[$\S$5.1]{BGo} for the contents of this subsection.\\

Let $h:M\rightarrow \mathbb{R}$ be a smooth function on $M$. We \textbf{assume} that
\begin{align}\begin{aligned}\label{e.634}
h \text{ is a Morse function along every fiber } Z.
\end{aligned}\end{align}
When $S$ is a point, such Morse function always exists (cf.  \cite[Lemma 1.5]{BruMa12}). But in the family case, the existence of such a function $h$
on $M$ verifying (\ref{e.634}) implies some topological obstructions of the fibration $\pi: M\rightarrow S$.

Let $\bf{B}$ be the set of the critical points of $h$ along the fibers. For $x\in \bf{B}$, let $\text{ind}(x)$ denote the index of $h$ at
$x$, i.e., the number of negative eigenvalues of the quadratic form $\big((d^{Z})^{2}h\big)(x)$ on $T_{x}Z$. Let ${\bf{B}}^{i}$ be the set of critical points of $h$
of index $i$ along the fibers $Z$.
Then $\bf{B}$, ${\bf{B}}^{i}$ are finite covers of $S$. We denote by $B$, $B^{i}$ the corresponding fibers.\\

We \textbf{assume} that for any $x\in {\bf{B}}_{\partial}:={\bf{B}}\cap X$
 the restriction of the quadratic form $(d^{Z})^{2}h$ to the normal direction of $X$ in $M$ verifies
\begin{align}\begin{aligned}\label{e.636}
\big((d^{Z})^{2}h\big)(x)|_{\bf{n}}>0.
\end{aligned}\end{align}

We also \textbf{assume} that $h$ is an even function in $x_{m}-$direction on $U_{\varepsilon}$, i.e.,
\begin{align}\begin{aligned}\label{e.650}
h(x',x_{m})=h(x',-x_{m})\quad \text{for} \quad (x',x_{m})\in X_{ [-\varepsilon,\varepsilon]}.
\end{aligned}\end{align}
Let $\nabla h\subset TZ$ be the gradient vector field of $h$ along the fiber
$Z$ with respect to certain vertical Riemannian metric, such that $(\nabla h)|_{X}\subset TY$. Note that there exists such a Morse function verifying (\ref{e.634}), (\ref{e.636}) and (\ref{e.650}), when $S$ is a point (cf. \cite[Lemma 1.5]{BruMa12}).

Consider the differential equation  $\frac{dx}{dt}=\mathscr{Y}(x):=-(\nabla h)(x) $ along the fibers $Z$.
This equation defines a group of diffeomorphisms $(\varphi_{t})_{t\in \mathbb{R}}$ of $Z$.
 If $x\in B$, put
\begin{align}\begin{aligned}\label{e.638}
&W_{Z}^{u}(x)=\{x'\in Z,\, \lim_{t\rightarrow -\infty}\varphi_{t}(x')=x\},\\
&W_{Z}^{s}(x)=\{x'\in Z,\, \lim_{t\rightarrow +\infty}\varphi_{t}(x')=x\}.
\end{aligned}\end{align}
The cells $W_{Z}^{u}(x)$ ($W_{Z}^{s}(x)$) are called the unstable (stable) cells along the fiber $Z$. They are embedded submanifolds of $Z$, and moreover,
\begin{align}\begin{aligned}\label{e.639}
W_{Z}^{u}(x)\simeq\mathbb{R}^{\text{ind}(x)},\quad W_{Z}^{s}(x)\simeq\mathbb{R}^{m-\text{ind}(x)}.
\end{aligned}\end{align}

As in \cite[$\S$5.5]{BGo}, we \textbf{assume} that $\mathscr{Y}$ verifies the Smale transversality conditions (cf. \cite{Smale}) along every fiber, i.e., for $x,x'\in B,\, x\neq x'$
\begin{align}\begin{aligned}\label{e.635}
W_{Z}^{u}(x) \text{ and }W_{Z}^{s}(x') \text{ intersect transversally.}
\end{aligned}\end{align}
If $x\in B$, set $T_{x}Z^{u}=T_{x}W^{u}(x)$ and $T_{x}Z^{s}=T_{x}W^{s}(x)$,
then we have
\begin{align}\begin{aligned}\label{e.641}
T_{x}Z=T_{x}Z^{u}\oplus T_{x}Z^{s}.
\end{aligned}\end{align}
Hence  $TZ^{u}$, $TZ^{s}$ are vector bundles on $\bf{B}$ with fiber, respectively, $T_{x}Z^{u}$, $T_{x}Z^{s}$ at $x\in \bf{B}$. Moreover, we have
\begin{align}\begin{aligned}\label{e.642}
TZ|_{\bf{B}}=TZ^{u}|_{\bf{B}}\oplus TZ^{s}|_{\bf{B}}.
\end{aligned}\end{align}

If $x\in B$, let $o^{u}_{x}$, $o^{s}_{x}$ be the orientation lines of $T_{x}Z^{u}$, $T_{x}Z^{s}$. Then $o^{u}_{x},\, o^{s}_{x}$ are $\mathbb{Z}_{2}-$lines.
By (\ref{e.639}), (\ref{e.641}), we can identify $o^{u}_{x},\, o^{s}_{x}$ with the orientation lines of $W^{u}_{Z}(x)$ and $W^{s}_{Z}(x)$.
We use $o^{u}$, $o^{s}$ to denote the orientation line bundles of $TZ^{u}$, $TZ^{s}$ over $\bf{B}$, whose fibers at $x\in \bf{B}$
are $o^{u}_{x}$, $o^{s}_{x}$ respectively.\\

Let $(F^{*},\nabla^{F^{*}})$ be the dual flat vector bundle of $(F,\nabla^{F})$. Set
 \begin{align}\begin{aligned}\label{6.47}
C_{\bullet}(W^{u}_{Z}, F^{*})=\bigoplus_{x\in B}F^{*}_{x}\otimes o^{u}_{x}\quad \text{and} \quad C_{i}(W^{u}_{Z}, F^{*})=\bigoplus_{x\in B^{i}}F^{*}_{x}\otimes o^{u}_{x},
\end{aligned}\end{align}
then $C_{\bullet}(W^{u}_{Z}, F^{*})$ is a $\mathbb{Z}-$graded flat vector bundle over $S$, with a
flat connection $\nabla^{C_{\bullet}(W^{u}_{Z}, F^{*})}$ induced by $\nabla^{F^{*}}$.

By the transversality conditions (\ref{e.635}), for $x,x'\in B$, if $\text{ind}(x')=\text{ind}(x)-1$, then $W^{u}_{Z}(x)\cap W^{s}_{Z}(x')$ consists of a finite
set $\Gamma(x,x')$ of integral curves $\gamma$ of the vector field $\mathscr{Y}|_{Z}$, with $\gamma_{-\infty}=x$, $\gamma_{+\infty}=x'$, along which $W^{u}_{Z}(x)$
and $W^{s}_{Z}(x')$ intersect transversally.

Let $x,x'\in B$, such that $\text{ind}(x')=\text{ind}(x)-1$. The orientation bundle of the orthogonal bundle $T^{\perp}W^{s}_{Z}(x')$ to $TW^{s}_{Z}(x')$
 in
$TZ|_{W^{s}_{Z}(x')}$ is canonically isomorphic to $o^{u}(x')$. Let $T'W^{s}_{Z}(x)$ be the orthogonal bundle to $\mathscr{Y}$ in $TW^{u}_{Z}(x)$. Its
orientation bundle $o(T'W^{s}_{Z}(x))$ is canonically isomorphic to $o^{u}_{x}$, so that $s\in o(T'W^{s}_{Z}(x))$ corresponds to $\mathscr{Y}\widehat{\otimes} s\in o^{u}_{x}$.
Since $T^{\perp}W^{s}_{Z}(x')$ and $T'W^{s}_{Z}(x)$ have the same orientation bundle, to $\gamma\in \Gamma(x,x')$, we can associate $n_{\gamma}(x,x')\in o^{u}_{x}\otimes
o^{u}_{x'}$.

By (\ref{e.639}), on $W^{u}_{Z}(x)$ the flat vector bundle $F^{*}$ can be canonically trivialized by parallel transport with respect to $\nabla^{F^{*}}$.
In particular, if $x,x'\in B$ are critical points such
that $\text{ind}(x')=\text{ind}(x)-1$ and if $\gamma\in \Gamma(x,x'),\, e^{*}\in F^{*}_{x}$,
let $\tau_{\gamma}(e^{*})\in F^{*}_{x'}$ be the parallel
transport of $e^{*}\in F^{*}_{x}$ along $\gamma$ by $\nabla^{F^{*}}$.

If $x\in B$, $e^{*}\in F^{*}_{x}$, $s\in o^{u}_{x}$, then the differential $\partial$
 is defined by
\begin{align}\begin{aligned}\label{6.51}
\partial(e^{*}\otimes s)=\sum_{\begin{array}{c}
                 x'\in B \\
                 \text{ind}(x')=\text{ind}(x)-1
                \end{array}
}\sum_{\gamma\in \Gamma(x, x')}\tau_{\gamma}(e^{*})\otimes n_{\gamma}(x, x')s.
\end{aligned}\end{align}
Then $\partial$ maps $C_{i}(W^{u}_{Z}, F^{*})$ into $C_{i-1}(W^{u}_{Z}, F^{*})$.
By a result of Thom \cite{Thom} and Smale \cite{Smale}, we have $\partial^{2}=0$ (cf. \cite[Thm. 5.1]{BGo}).
From the definitions of $\nabla^{C_{\bullet}(W^{u}_{Z}, F^{*})}$ and $\partial$, we have
\begin{align}\begin{aligned}\label{e.645}
[\nabla^{C_{\bullet}(W^{u}_{Z}, F^{*})},\partial]=0.
\end{aligned}\end{align}

 Let $(C^{\bullet}(W^{u}_{Z}, F),\nabla^{C^{\bullet}(W^{u}_{Z}, F)}, \widetilde{\partial})$ be the dual complex of $(C_{\bullet}(W^{u}_{Z}, F^{*}),\nabla^{C_{\bullet}(W^{u}_{Z}, F^{*})}, \partial)$, then by (\ref{6.47}) we have
\begin{align}\begin{aligned}\label{6.8}
C^{\bullet}(W^{u}_{Z}, F)=\bigoplus_{x\in B}F_{x}\otimes o^{u}_{x}\quad \text{and} \quad C^{i}(W^{u}_{Z}, F)=\bigoplus_{x\in B^{i}}F_{x}\otimes o^{u}_{x}.
\end{aligned}\end{align}
Then $(C^{\bullet}(W^{u}_{Z},F),\nabla^{C^{\bullet}(W^{u}_{Z}, F)},\widetilde{\partial})$ is the Thom-Smale complex of flat vector bundles attached to $\mathscr{Y}|_{Z}$ along the fiber $Z$.
By duality, we get
$$\widetilde{\partial}^{2}=0,\quad \quad[\nabla^{C^{\bullet}(W^{u}_{Z}, F)},\widetilde{\partial}]=0,$$
hence we see that
$
(C^{\bullet}(W^{u}_{Z}, F),\nabla^{C^{\bullet}(W^{u}_{Z}, F)},\widetilde{\partial})
$
defines a complex of flat vector bundles over $S$ (see (\ref{5.45}), (\ref{5.47})). In particular, $A'=\widetilde{\partial}+\nabla^{C^{\bullet}(W^{u}_{Z}, F)}$ is a flat superconnection of
total degree 1 on $C^{\bullet}(W^{u}_{Z}, F)$. Let $h^{C^{\bullet}(W^{u}_{Z}, F)}$ be the Hermitian metric
on $C^{\bullet}(W^{u}_{Z}, F)$ induced by $h^{F}$.
Then we define the combinatorial torsion form $T_{f}(A^{C^{\bullet}(W^{u}_{Z}, F)}, h^{C^{\bullet}(W^{u}_{Z}, F)})$ over $S$
by Definition \ref{d.9}.

By our assumptions (\ref{e.634}), (\ref{e.636}), we have that $((d^{Z})^{2}h(x))|_{\bf{n}}>0$ for $x\in {\bf{B}}_{\partial}$ and $\mathscr{Y}|_{X}\in TY$,
hence $\mathscr{Y}|_{X}$ verifies the Smale transversality conditions (\ref{e.635}). Set
 \begin{align}\begin{aligned}\label{6.48}
C_{\bullet}(W^{u}_{Y}, F^{*})=\bigoplus_{y\in B_{\partial}}F^{*}_{y}\otimes o^{u}_{y}\quad \text{and} \quad C_{i}(W_{Y}^{u}, F^{*})=\bigoplus_{y\in B_{\partial}^{i}}F^{*}_{y}\otimes o^{u}_{y},
\end{aligned}\end{align}
where $B_{\partial}$ is the fiber of ${\bf{B}}_{\partial}$. Then $(C_{\bullet}(W^{u}_{Y}, F^{*}), \nabla^{C_{\bullet}(W^{u}_{Y}, F^{*})},\partial)$ is a sub-complex of flat vector bundles (\ref{6.47}).

\subsection{Torsion forms associated to the spectral sequences of double complex}\label{ss2.2}
\begin{defn}\label{d.19}
Let $c^{p, q}, 0\leq p \leq 2, 0\leq q \leq m$, be vector spaces of finite dimension. Let
\begin{align}\begin{aligned}\label{e.369}
\partial:c^{p, q}\rightarrow c^{p, q+1}, \quad v:c^{p, q}\rightarrow c^{p+1, q}
\end{aligned}\end{align}
be some linear maps such that
\begin{align}\begin{aligned}\label{e.368}
\partial ^{2}=v^{2}=0 \text{ and } \partial v+ v\partial =0.
\end{aligned}\end{align}
Then $(c^{p,q},\partial, v)$ forms a double complex of vector spaces. Let $\{C^{*}, D\}$ be its total complex defined by
\begin{align}\begin{aligned}\label{e.370}
 C^{n}=\bigoplus_{p+q=n}c^{p, q}, \quad D=\partial + v,
\end{aligned}\end{align}
whose the cohomology groups are denoted by $H^{*}(C^{*})$.
\end{defn}

 Let $\{'E_{r}, 'd_{r}\}$ be the spectral sequence (cf. \cite[$\S$3.5]{Gri}) associated to the double complex $c^{p, q}$
 with the filtration on the total complex defined by
 \begin{align}\begin{aligned}\label{e.371}
 'F^{p}C^{n}=\bigoplus_{\mbox{\tiny %
$ \begin{array}{c}
 p'+q=n \\
 p'\geq p
 \end{array}
$ }}c^{p', q},
\end{aligned}\end{align}
and let $\{''E_{r}, ''d_{r}\}$\index{$\{''E_{r}, ''d_{r}\}$} be the spectral sequence associated to the double complex $c^{p, q}$ with the filtration
 \begin{align}\begin{aligned}\label{e.372}
 ''F^{q}C^{n}=\bigoplus_{\mbox{\tiny %
$
 \begin{array}{c}
 p+q''=n \\
  q''\geq q
 \end{array}$ }}c^{p, q''}\, .
\end{aligned}\end{align}
The filtration $F^{p}C^{*}$ on $C^{*}$ also induces a filtration $F^{p}H^{*}(C^{*})$ on the cohomology $H^{*}(C^{*})$ (cf. \cite[$\S$3.5]{Gri}). The
associated graded cohomology is
 \begin{align}\begin{aligned}\label{e.820}
 {\rm Gr}H^{*}(C^{*})=\bigoplus_{p,q}{\rm Gr}^{p}H^{q}(C^{*}),\quad \text{with}\quad\text{Gr}^{p}H^{q}(C^{*}):=\frac{F^{p}H^{q}(C^{*})}{F^{p+1}H^{q}(C^{*})}.
\end{aligned}\end{align}
 The spectral sequence $E_{k}$ converges to $E_{\infty}=E_{k_{0}}$ with
 \begin{align}\begin{aligned}\label{e.789}
 E^{p,q}_{\infty}=\text{Gr}^{p}H^{p+q}(C^{*}).
\end{aligned}\end{align}

In our situation there exists a natural double complex of flat vector bundles, that is

 \begin{align}\begin{aligned}\label{e.373}
\xymatrix{
& && &\\
0\ar[r]&C^{p}(W^{u}_{Z_{2}}/W^{u}_{Y}, F)\ar[r]^{\widetilde{v}^{0, p}}\ar[u]^{\widetilde{\partial}^{0, p}}
&C^{p}(W^{u}_{Z}, F) \ar[r]^{\widetilde{v}^{1, p}}\ar[u]^{\widetilde{\partial}^{1, p}}
&C^{p}(W^{u}_{Z_{1}}, F)\ar[r]\ar[u]^{\widetilde{\partial}^{2, p}}
&0\, .}
\end{aligned}\end{align}
Since every row of the double complex (\ref{e.373}) is exact, we have
\begin{align}\begin{aligned}\label{e.374}
H^{*}(C^{*})=0.
\end{aligned}\end{align}
Hence by (\ref{e.789}) we find
 \begin{align}\begin{aligned}\label{e.377}
 {'E}_{\infty}={{'E}}_{3}=0.\end{aligned}\end{align}

For the filtration (\ref{e.371}), the first term ${'E}_{0}^{p, q}$ of the spectral sequence is
\begin{align}\begin{aligned}\label{e.375}
{'E}_{0}^{p, q}=\frac{{'F}^{p}C^{p+q}}{{'F}^{p+1}C^{p+q}}=c^{p, q},
\end{aligned}\end{align}
whose differential ${'d}_{0}$ induced by $D$ is ${'d}_{0}=\widetilde{\partial}$.  For $0\leq q\leq m$, the next term ${'E}_{1}^{p, q}$ with ${'d}_{1}=\widetilde{v}$ is given by
\begin{align}\begin{aligned}\label{e.376}
{'E}_{1}^{0, q}=H^{q}(Z_{2}, Y, F), \quad
{'E}_{1}^{1, q}=H^{q}(Z, F), \quad
{'E}_{1}^{2, q}=H^{q}(Z_{1}, F),
\end{aligned}\end{align}
which form the following complex
\begin{align}\begin{aligned}\label{e.378}
({'E}^{*, q}_{1}, {{'d}_{1}}):\quad 0 \rightarrow {{'E}_{1}^{0, q}} \overset{{'d}_{1}}\rightarrow {{'E}_{1}^{1, q}} \overset{{'d}_{1}}\rightarrow {'E}_{1}^{2, q}\rightarrow 0.\end{aligned}\end{align}
The last terms are
\begin{align}\begin{aligned}\label{e.379}
{'E}_{2}^{0, q}=\Ker({{'d}}_{1})\cap H^{q}(Z_{2}, Y, F), \quad {'E}_{2}^{1, q}={{'E}}_{3}^{1, q}, \quad{'E}_{2}^{2, q}=H^{q}(Z_{1}, F)/\Im({{'d}_{1}}),
\end{aligned}\end{align}
which form the complexe
\begin{align}\begin{aligned}\label{e.380}
0\rightarrow {{'E}}_{2}^{0, q}\overset{{'d}_{2}} \longrightarrow {{'E}} _{2}^{2, q-1}\rightarrow 0,
\end{aligned}\end{align}
 where ${'d}_{2}$ is the coboundary map. From (\ref{e.377}), we know that the sequence (\ref{e.380}) is exact and
 \begin{align}\begin{aligned}\label{e.405}
 {'E}_{2}^{1, q}={{'E}}_{3}^{1, q}=0.
\end{aligned}\end{align}

For the filtration (\ref{e.372}), the first term of $\{'{'E}_{r}, {''d}_{r}\}$ is given by
 \begin{align}\begin{aligned}\label{e.382}
({''E}_{0}, {''d}_{0}):\quad 0 \rightarrow C^{p}(W^{u}_{Z_{2}}/W^{u}_{Y}, F)\overset{\widetilde{v}^{0, p}}\rightarrow
C^{p}(W^{u}_{Z}, F) \overset{\widetilde{v}^{1, p}}\rightarrow
C^{p}(W^{u}_{Z_{1}}, F)
\rightarrow 0, \end{aligned}\end{align}
 which is exact and split, hence we have
 \begin{align}\begin{aligned}\label{e.385}
 {''E}_{\infty}={''E}_{1}=0 .\end{aligned}\end{align}

We use $(\mathscr{H},\nabla^{\mathscr{H}})$ to denote the long exact sequence of flat vector bundles associated to (\ref{e.373}) with the $L^{2}-$metric $h_{L^{2}}^{\mathscr{H}}$ induced by Hodge-De Rham theory (see Section \ref{ss1.4}), that is
\begin{align}\begin{aligned}\label{5.15}
(\mathscr{H},\nabla^{\mathscr{H}},\delta):\cdots \longrightarrow H^{i}(Z, F)\overset{\delta}\longrightarrow H^{i}(Z_{1}, F) \overset{\delta}\longrightarrow H^{i+1}(Z_{2}, Y, F) \overset{\delta}\longrightarrow\cdots\,.
\end{aligned}\end{align}
Then we can associate a torsion form $T_{f}(A^{\mathscr{H}}, h_{L^{2}}^{\mathscr{H}})$ to $(\mathscr{H},\nabla^{\mathscr{H}},\delta)$ as in Definition \ref{d.9}.

\begin{lemma}\label{l.37}The following identity holds in $Q^{S}/Q^{S, 0}$
 \begin{align}\begin{aligned}\label{5.21}
T_{f}(A^{\mathscr{H}}, h_{L^{2}}^{\mathscr{H}})=T_{f}\left(A^{{'E}_{1}}_{1}, h^{{'E}_{1}}_{L^{2}}\right)+T_{f}\left(A^{{'E}_{2}}_{2}, h^{{'E}_{2}}_{L^{2}}\right).
\end{aligned}\end{align}
\end{lemma}
\begin{proof} By (\ref{e.376})---(\ref{e.382}), we have the following exact sequences:
 \begin{align}\begin{aligned}\label{e.790}
  &A_{0}: \quad 0\rightarrow H^{0}(Z_{2}, Y, F)
 \overset{{'d}_{1}}\rightarrow
H^{0}(Z, F)\overset{{'d}_{1}}\rightarrow H^{0}(Z_{1}, F)\rightarrow H^{0}(Z_{1}, F)/\Im({{'d}_{1}})\rightarrow 0,\\
 &A_{q\geq 1}: \quad 0\rightarrow \Ker({{'d}}_{1})\cap H^{q}(Z_{2}, Y, F)\rightarrow H^{q}(Z_{2}, Y, F)\\
& \qquad \qquad \qquad \qquad\overset{{'d}_{1}}\rightarrow
H^{q}(Z, F)\overset{{'d}_{1}}\rightarrow H^{q}(Z_{1}, F)\rightarrow H^{q}(Z_{1}, F)/\Im({{'d}_{1}})\rightarrow 0,\\
&B_{q\geq 1}:\quad 0\rightarrow \Ker({{'d}}_{1})\cap H^{q}(Z_{2}, Y, F)\overset{{'d}_{2}}\rightarrow H^{q-1}(Z_{1}, F)/\Im({{'d}_{1}})\rightarrow 0.
\end{aligned}\end{align}
Their torsion forms will be denoted by $T_{f}(A_{q})$ and $T_{f}(B_{q})$ for short with respect to $L_{2}-$metrics. For $1\leq q\leq m$, let
\begin{align}\begin{aligned}\label{e.793}
&B_{q}^{-1}:\quad 0\rightarrow H^{q-1}(Z_{1}, F)/\Im({{'d}_{1}})\overset{{'d}^{-1}_{2}}\rightarrow \Ker({{'d}}_{1})\cap H^{q}(Z_{2}, Y, F) \rightarrow 0,\\
&C_{q}:=A_{q}\circ B_{q}^{-1} :\quad 0\rightarrow H^{q-1}(Z_{1}, F)/\Im({{'d}_{1}}) \rightarrow H^{q}(Z_{2}, Y, F)\\
& \qquad \qquad \qquad\overset{{'d}_{1}}\rightarrow
H^{q}(Z, F)\overset{{'d}_{1}}\rightarrow H^{q}(Z_{1}, F)\rightarrow H^{q}(Z_{1}, F)/\Im({{'d}_{1}})\rightarrow 0,
 \end{aligned}\end{align}
 then by Lemma \ref{l.36}, we get in $Q^{S}/Q^{S, 0}$
 \begin{align}\begin{aligned}\label{e.792}
 T_{f}(B_{q})=-T_{f}(B^{-1}_{q})\quad \text{and} \quad T_{f}(C_{q})=T_{f}(A_{q})+T_{f}(B^{-1}_{q}).
  \end{aligned}\end{align}
 By functoriality of torsion form (cf. \cite[Lemma 3.1 (d)]{Ma02}), Lemma \ref{l.36}, (\ref{e.378}), (\ref{e.380}) and (\ref{e.792}), we have in $Q^{S}/Q^{S, 0}$
 \begin{align}\begin{aligned}\label{e.791}
 &T_{f}\left(A^{{'E}_{1}}_{1}, h^{{'E}_{1}}_{L^{2}}\right)=\sum_{q=0}^{m}(-1)^{q}T_{f}({'E}^{*, q}_{1})=T_{f}(A_{0})+\sum_{q=1}^{m}(-1)^{q-1}T_{f}(A_{q}),\\
 &T_{f}\left(A^{{'E}_{2}}_{2},h^{{'E}_{2}}_{L^{2}}\right)=\sum_{q=1}^{m}(-1)^{q}T_{f}(B_{q})=\sum_{q=1}^{m}(-1)^{q-1}T_{f}(B^{-1}_{q}).
\end{aligned}\end{align}
Hence by (\ref{e.792}) and (\ref{e.791}) the right side of equation (\ref{5.21}) in $Q^{S}/Q^{S, 0}$ is equal to
\begin{align}\begin{aligned}\label{e.794}
T_{f}(A_{0})&+\sum_{q=1}^{m}(-1)^{q-1}T_{f}(A_{q})+\sum_{q=1}^{m}(-1)^{q-1}T_{f}(B^{-1}_{q})\\
&=T_{f}(A_{0})+\sum_{q=1}^{m}(-1)^{q-1}T_{f}(C_{q})=T_{f}(A^{\mathscr{H}}, h_{L^{2}}^{\mathscr{H}}),
\end{aligned}\end{align}
for the last equality we have used Lemma \ref{l.36} and the fact that $\mathscr{H}=C_{m}\circ C_{m-1}\circ\cdots\circ C_{1}\circ A_{0}$.
The proof of Lemma \ref{l.37} is completed.
\end{proof}

  Let $h^{E_{k}}$ be the metric on $E_{k}$ induced by the metric $h^{c^{p, q}}$ on $c^{p, q}$ through the finite dimensional Hodge theory. Recall that the filtration on $H^{q}(C^{*})$ is defined in (\ref{e.820}) and we have  $E^{p,q}_{\infty}=\text{Gr}^{p}H^{p+q}(C^{*})$ (see (\ref{e.789})).
We assume that $\widetilde{h}^{H(C^{*})}$ is a metric which is adapted to the induced filtration on $H^{*}(C^{*})$ (see \cite[Def. 7.35]{Goette03}).
Let $T_{f}\big(A^{E_{k}}_{k}, h^{E_{k}}\big)$ be the torsion form associated to the complexes of flat vector bundles $\{E_{k}, d_{k}\}$ by
Definition \ref{d.9}.
Let $\widetilde{f}\big(\nabla^{H(C^{*})}, \widetilde{h}^{H(C^{*})}, h^{H(C^{*})}\big)$ and $\widetilde{f}\big(\nabla^{E_{\infty}},
\widetilde{h}^{H(C^{*})}, h^{E_{\infty}}\big)$ be the torsion forms associated to filtered flat complex vector bundle defined in
\cite[Def. 3.1]{Ma02}.
Goette proved the following theorem \cite[Thm. 7.37]{Goette03}, which is a finite dimensional version of a result of Ma \cite[Thm. 0.1]{Ma02} in an infinite-dimensional setting.
\begin{thm}\label{t.15}
The following identity holds in $Q^{S}/Q^{S, 0}$
\begin{align}\begin{aligned}\label{e.381}
T_{f}\big(A^{C^{*}}, h^{C^{*}}\big)+\widetilde{f}&\big(\nabla^{H(C^{*})}, \widetilde{h}^{H(C^{*})}, h^{H(C^{*})}\big) \\
&=T_{f}\big(A^{E_{0}}_{0}, h^{E_{0}}\big)+\sum_{k=1}^{k_{0}}T_{f}\big(A^{E_{k}}_{k}, h^{E_{k}}\big)+\widetilde{f}\big(\nabla^{E_{\infty}},
\widetilde{h}^{H(C^{*})}, h^{E_{\infty}}\big),
\end{aligned}\end{align}
where $f(x)=xe^{x^{2}}$ and $k_{0}$ is chosen such that $E_{\infty}={E_{k_{0}+1}}$.
\end{thm}
\begin{rem}The holomorphic version of Theorem \ref{t.15} has been established by Ma \cite[Thm. 1.2]{Ma00}. In our application of Theorem \ref{t.15} in this paper, we always have $H(C^{*})=0$, therefore by (\ref{e.789}) we have $E_{\infty}=0$.
Thus we can assume that there exists a metric adapted to the induced filtration on $H^{*}(C^{*})$.
\end{rem}

By applying Theorem \ref{t.15} for the filtration (\ref{e.371}), we get from (\ref{e.374}), (\ref{e.377}) that
\begin{align}\begin{aligned}\label{e.384}
T_{f}\big(A^{C^{*}}, h^{C^{*}}\big)=T_{f}\left(A^{{'E}_{0}}_{0}, h^{{'E}_{0}}\right)+T_{f}\left(A^{{'E}_{1}}_{1}, h^{{'E}_{1}}\right)
+T_{f}\left(A^{{'E}_{2}}_{2}, h^{{'E}_{2}}\right).
\end{aligned}\end{align}

Since the sequence (\ref{e.382}) is exact and split, by \cite[Thm. A1.1]{BL} we have
\begin{align}\begin{aligned}\label{e.386}
 T_{f}\big(A^{{''E}_{0}}_{0}, h^{{''E}_{0}}\big)=0.
\end{aligned}\end{align}
 By applying Theorem \ref{t.15} for the filtration (\ref{e.372}), we get from (\ref{e.385}) and (\ref{e.386}) that the torsion form of the total complex $\{C^{*}, D\}$ is vanished, i.e.,
\begin{align}\begin{aligned}\label{e.387}
T_{f}\left(A^{C^{*}}, h^{C^{*}}\right)=0.
\end{aligned}\end{align}

 Let $h^{E_{1}}_{L^{2}},\,h^{E_{2}}_{L^{2}}$ be the $L^{2}$-metric induced by
Hodge-de Rham theory (see Section \ref{ss1.4}). By Theorem \ref{t.16} and (\ref{e.377}), we have in $Q^{S}/Q^{S, 0}$
\begin{align}\begin{aligned}\label{5.9}
T_{f}\left(A^{{'E}_{1}}_{1}, h^{{'E}_{1}}\right)-&T_{f}\left(A^{{'E}_{1}}_{1}, h^{{'E}_{1}}_{L^{2}}\right)\\
&=\widetilde{f}\left(\nabla^{{'E}_{1}},
h^{{'E}_{1}}_{L^{2}}, h^{{'E}_{1}}\right)-\widetilde{f}\left(\nabla^{{'E}_{2}},
h^{{'E}_{2}}_{L^{2}}, h^{{'E}_{2}}\right),\end{aligned}\end{align}
and
\begin{align}\begin{aligned}\label{5.35}
T_{f}\left(A^{{'E}_{2}}_{2}, h^{{'E}_{2}}\right)-T_{f}\left(A^{{'E}_{2}}_{2}, h^{{'E}_{2}}_{L^{2}}\right)
=\widetilde{f}\left(\nabla^{{'E}_{2}},
h^{{'E}_{2}}_{L^{2}}, h^{{'E}_{2}}\right).\end{aligned}\end{align}
From (\ref{e.384}), (\ref{e.387}), (\ref{5.9}) and (\ref{5.35}), it follows
\begin{align}\begin{aligned}\label{e.413}
0=T_{f}\left(A^{{'E}_{0}}_{0}, h^{{'E}_{0}}\right)+&T_{f}\left(A^{{'E}_{1}}_{1}, h^{{'E}_{1}}_{L^{2}}\right)\\
&+T_{f}\left(A^{{'E}_{2}}_{2}, h^{{'E}_{2}}_{L^{2}}\right)
+\widetilde{f}\left(\nabla^{{'E}_{1}},
h^{{'E}_{1}}_{L^{2}}, h^{{'E}_{1}}\right).\end{aligned}\end{align}

\subsection{Double formula for the analytic torsion form}\label{ss2.4}
In Sections \ref{ss2.4} and \ref{ss2.5}, we adopt the notations of Section \ref{ss1.1}. Recall that the double fibration $(\overline{M},\phi)$ is defined in Definition \ref{d.39}.
Then $(\overline{M},\overline{F})$ admits a $\mathbb{Z}_{2}$-action where $\mathbb{Z}_{2}:=\{\Id, \phi\}$.

 For $g\in \mathbb{Z}_{2}$, set
\begin{align}\begin{aligned}\label{e.647}
\chi_{g}(\overline{Z})=\sum_{p=0}^{m}(-1)^{p}\tr^{H^{p}(\overline{Z},\mathbb{R})}[g],\quad \chi_{g}'(\overline{Z}, \overline{F})=\sum_{p=0}^{m}(-1)^{p}p\tr^{H^{p}(\overline{Z},\overline{F})}[g].
\end{aligned}\end{align}
The $\mathbb{Z}_{2}-$equivariant analytic torsion form was defined by Bismut-Goette in \cite[Def. 3.31]{BGo}.
\begin{defn}\label{d.20} For $g\in \mathbb{Z}_{2}$, we define
\begin{align}\begin{aligned}\label{6.1}
\mathscr{T}_{g}(T^{H}\overline{M}, &g^{T\overline{Z}}, h^{\overline{F}}):=-\int_{0}^{+\infty}\left[f_{g}^{\wedge}(\overline{C}'_{t}, h^{\Omega^{\bullet}(\overline{Z},F|_{\overline{Z}})})-\frac{\chi_{g}'(\overline{Z}, \overline{F})}{2}f'(0)\right.\\
&\left.-\left(\frac{1}{4}m\rk(\overline{F})\chi_{g}(\overline{Z})-\frac{\chi_{g}'(\overline{Z}, \overline{F})}{2}\right)f'(\frac{i\sqrt{t}}{2})\right]\frac{dt}{t}.
\end{aligned}\end{align}\index{$\mathscr{T}_{g}(T^{H}\overline{M}, g^{T\overline{Z}}, h^{\overline{F}})$}
\end{defn}

Analogous to \cite[Prop. 2.1]{BruMa12}, we have the double formula for analytic torsion form.
\begin{prop}\label{p6.1}{\rm (Double formula for the analytic torsion form)} For $g\in \mathbb{Z}_{2}$, we have
\begin{align}\begin{aligned}\label{6.3}
\mathscr{T}_{g}(T^{H}\overline{M}, g^{T\overline{Z}}, h^{\overline{F}})=\mathscr{T}_{{\rm abs}}(T^{H}M, g^{TZ}, h^{F})+\chi(g)\mathscr{T}_{{\rm rel}}(T^{H}M, g^{TZ}, h^{F}),
\end{aligned}\end{align}
where $\chi$ is the nontrivial character of $\mathbb{Z}_{2}$.
\end{prop}
\begin{proof}
The equation (\ref{6.3}) is obtained by \cite[(2.14)]{BruMa12}, Lemmas \ref{l1.1}, \ref{l.46}, Definitions \ref{d.6}, \ref{d.20} and (\ref{e.568}).
\end{proof}

\subsection{Double formula for the combinatorial torsion form}\label{ss2.5}
 By our assumption (\ref{e.650}), the fiberwise Morse function $h$ induces naturally a $\mathbb{Z}_{2}-$invariant Morse function $\overline{h}$ on $\overline{M}$,
 such that $\overline{h}|_{M}=h$. Then the Thom-Smale complex of flat vector bundles
\begin{align}\begin{aligned}\label{6.4}
\big(C^{\bullet}(W^{u}_{\overline{Z}}, \overline{F}), \nabla^{C^{\bullet}(W^{u}_{\overline{Z}}, \overline{F})},\widetilde{\partial}\big)
\end{aligned}\end{align}
 admits a natural $\mathbb{Z}_{2}-$action of
$\phi$. By \cite[Def. 1.29]{BGo}, we associate a $\mathbb{Z}_{2}-$equivariant torsion form $T_{f, g}(A^{C^{\bullet}(W^{u}_{\overline{Z}}, \overline{F})}, h^{C^{\bullet}(W^{u}_{\overline{Z}}, \overline{F})})$ to (\ref{6.4}).

The bundle $C^{\bullet}(W^{u}_{\overline{Z}}, \overline{F})$ with the action of $\phi$ has an eigen-decomposition:
\begin{align}\begin{aligned}\label{6.5}
C^{\bullet}(W^{u}_{\overline{Z}}, \overline{F})=C^{\bullet}(W^{u}_{\overline{Z}}, \overline{F})^{+}\bigoplus C^{\bullet}(W^{u}_{\overline{Z}}, \overline{F})^{-},
\end{aligned}\end{align}
where $C^{\bullet}(W^{u}_{\overline{Z}}, \overline{F})^{\pm}$ is the subbundle of
$C^{\bullet}(W^{u}_{\overline{Z}}, \overline{F})$ on which $\phi$ acts by multiplication by
$\pm 1$. As $A^{C^{\bullet}(W^{u}_{\overline{Z}}, \overline{F})}$ and $h^{C^{\bullet}(W^{u}_{\overline{Z}}, \overline{F})}$
induce the corresponding flat superconnections and metrics on the sub-complex of flat vector bundles
$(C^{\bullet}(W^{u}_{\overline{Z}}, \overline{F})^{\pm}, \widetilde{\partial})$, then we can define their combinatorial torsion forms
 $T_{f}(A^{C^{\bullet}(W^{u}_{\overline{Z}}, \overline{F})^{\pm}}, h^{C^{\bullet}(W^{u}_{\overline{Z}}, \overline{F})^{\pm}})$ by Definition \ref{d.9}.

\begin{prop}\label{p6.2}The following identity holds in $Q^{S}/Q^{S, 0}$, for $g\in \mathbb{Z}_{2}$
\begin{align}\begin{aligned}\label{6.6}
&T_{f, g}\big(A^{C^{\bullet}(W^{u}_{\overline{Z}}, \overline{F})}), h^{C^{\bullet}(W^{u}_{\overline{Z}}, \overline{F})}\big)\\
&=T_{f}\big(A^{C^{\bullet}(W^{u}_{\overline{Z}}, \overline{F})^{+}}, h^{C^{\bullet}(W^{u}_{\overline{Z}}, \overline{F})^{+}}\big)
+\chi(g)T_{f}\big(A^{C^{\bullet}(W^{u}_{\overline{Z}}, \overline{F})^{-}}, h^{C^{\bullet}(W^{u}_{\overline{Z}}, \overline{F})^{-}}\big),
\end{aligned}\end{align}
where $\chi$ is the nontrivial character of $\mathbb{Z}_{2}$.
\end{prop}
\begin{proof}
The formula (\ref{6.6}) follows from \cite[Def. 1.29]{BGo} and the decomposition (\ref{6.5}).
\end{proof}

\section{Gluing formula with the existence of fiberwise Morse function}\label{s.3}
In this section, we will prove Theorem \ref{t.14} under the existence of fiberwise Morse function
by using the ideal similar to that of Br\"{u}ning and Ma \cite{BruMa12}.

Roughly speaking, our strategy to prove Theorem \ref{t.14} is: First, we establish some comparison theorems between analytic torsion forms and
 combinatorial torsion forms on $M$
(resp. $M_{1}$ and $M_{2}$). Second, we prove a ``gluing" formula for the three combinatorial torsion forms introduced in Section \ref{s.2}.
In one-point case, i.e., $S$ is a point, this relation is a trivial result (cf. \cite[(3.71)]{BruMa12}). In the family case, the gluing relation is non-trivial, which is the main difficulty in generalizing the result of Br\"{u}ning-Ma to
the family case. Last, we obtain the gluing formula of analytic torsion forms through that of the combinatorial torsion form in certain sense.

In Section \ref{ss2.3}, we state a key result, Theorem \ref{t.17}, which will be proved in the next subsections.
Then we establish the gluing formula (\ref{e.355}) of analytic torsion forms.
The remaining sections are all devoted to prove Theorem \ref{t.17}.
In Section \ref{ss2.6}, we prove some intermediate lemmas.
In Section \ref{ss2.7}, we explicitly compute certain combinatorial torsion forms $T_{f}\big(A^{E_{k}}_{k}, h^{E_{k}}\big)$ associated to certain spectral sequences.
These terms will be responsible for the appearance of the summand $\frac{1}{2}\log 2\rk(F)\chi(Y)$ in (\ref{e.355}).
In Section \ref{ss2.8}, we apply a result due to Bismut and Goette \cite{BGo} for the equivariant analytic torsion form
 to link the analytic torsion forms and their combinatorial counterparts.

\subsection{The main result}\label{ss2.3}

We state a theorem which will be proved in the next subsections.
\begin{thm}\label{t.17}The following identity holds in $Q^{S}/Q^{S, 0}$
\begin{align}\begin{aligned}\label{5.11}
\mathscr{T}&(T^{H}M, g^{TZ}, h^{F})-\mathscr{T}_{\rm abs}(T^{H}M_{1}, g^{TZ_{1}}, h^{F})-\mathscr{T}_{\rm rel}(T^{H}M_{2}, g^{TZ_{2}}, h^{F})\\
&=\frac{\log 2}{2}\rk(F)\chi(Y)-T_{f}\left(A^{{'E}_{0}}_{0}, h^{{'E}_{0}}\right)-\widetilde{f}\left(\nabla^{{'E}_{1}},
h^{{'E}_{1}}_{L^{2}}, h^{{'E}_{1}}\right).\end{aligned}\end{align}
\end{thm}

From (\ref{e.413}), Lemma \ref{l.37} and Theorem \ref{t.17}, we get
\begin{thm}\label{t5.2}The following identity holds in $Q^{S}/Q^{S, 0}$
\begin{align}\begin{aligned}\label{5.12}
\mathscr{T}(T^{H}M, g^{TZ}, h^{F})-&\mathscr{T}_{\rm abs}(T^{H}M_{1}, g^{TZ_{1}}, h^{F})-\mathscr{T}_{\rm rel}(T^{H}M_{2}, g^{TZ_{2}}, h^{F})\\
&=\frac{\log 2}{2}\rk(F)\chi(Y)+T_{f}(A^{\mathscr{H}}, h_{L^{2}}^{\mathscr{H}}).\end{aligned}\end{align}
\end{thm}
 The next subsections will be devoted to prove Theorem \ref{t.17}.

\subsection{Some intermediate results}\label{ss2.6}
Let $j_{1}:M_{1}\hookrightarrow \overline{M_{1}}$ be the including map of $M_{1}$ into its double $\overline{M_{1}}$ (see Def. \ref{d.39}) as one copy and
$j_{2}$ be that of $M_{1}$ into $\overline{M_{1}}$ as another copy. We identify $Z_{1}$ with $j_{1}(Z_{1})$ and denote $j_{2}(Z_{1})$ by $Z'_{1}$, then we have
 \begin{align}\begin{aligned}\label{e.417}
 j_{1}=\phi\circ j_{2}\,, \quad j_{2}= \phi\circ j_{1},\quad \overline{Z_{1}}=Z_{1}\cup_{X} Z'_{1}.
\end{aligned}\end{align}
Let $h_{1}=h|_{M_{1}}$ and $\overline{h_{1}}$ be the corresponding $\mathbb{Z}_{2}$-invariant Morse function on $\overline{M_{1}}$ (see Section \ref{ss2.4}). Denote by $\mathbb{C}^{+}$, $\mathbb{C}^{-}$ the trivial and the nontrivial one dimensional complex $\mathbb{Z}_{2}-$representation, respectively, and let $1_{\mathbb{C}^{+}}$, $1_{\mathbb{C}^{-}}$ be their unit elements. We define a $\mathbb{Z}_{2}$-equivariant short exact sequence of Thom-Smale complexes of flat vector bundles (cf. \cite[(2.7)]{BruMa12})
\begin{align}\begin{aligned}\label{6.49}
0 \longrightarrow C^{\bullet}(W_{\overline{Z_{1}}}^{u}, \overline{F})^{+}
\overset{\psi^{+}_{1}}\rightarrow C^{\bullet}(W^{u}_{Z_{1}}, F)\otimes \mathbb{C}^{+}\rightarrow 0,
\end{aligned}\end{align}
given by (cf. (\ref{6.8}))
\begin{align}\begin{aligned}\label{6.50}
\psi^{+}_{1}\left(\mathfrak{a}^{*}\right)
=\frac{\sqrt{2}}{2}\Big(j_{1}^{*}(\mathfrak{a}^{*}|_{Z_{1}}) +j_{2}^{*}(\mathfrak{a}^{*}|_{Z'_{1}})\Big)\otimes 1_{\mathbb{C}^{+}}.
\end{aligned}\end{align}

For the fibration $M_{2}$, we have the following $\mathbb{Z}_{2}$-equivariant short exact sequence:
\begin{align}\begin{aligned}\label{6.49}
 0\rightarrow C^{\bullet}(W^{u}_{Z_{2}}/W^{u}_{Y}, F)\otimes \mathbb{C}^{-}\overset{\psi^{-}_{2}} \longrightarrow C^{\bullet}(W_{\overline{Z_{2}}}^{u}, \overline{F})^{-}\rightarrow 0,
\end{aligned}\end{align}
given by
 \begin{align}\begin{aligned}\label{e.811}
\psi^{-}_{2}\left(\mathfrak{b}^{*}\otimes 1_{\mathbb{C}^{-}}\right)
=\frac{\sqrt{2}}{2}\Big((j_{1}^{-1})^{*}\mathfrak{b}^{*} -(j_{2}^{-1})^{*}\mathfrak{b}^{*}\Big),
\end{aligned}\end{align}
where $j_{1},j_{2}$ are defined in the same way as for $M_{1}$.

\begin{lemma}\label{l.38}
The following diagrams of flat vector bundles
\begin{align}\begin{aligned}\label{e.422}
\xymatrix{
&&&\\
0\ar[r]&C^{\bullet}(W^{u}_{\overline{Z_{1}}}, \overline{F})^{+}\ar[r]^{ \psi^{+}_{1}}\ar[u]^{\widetilde{\overline{\partial}}}&C^{\bullet}(W^{u}_{Z_{1}}, F)\otimes \mathbb{C}^{+}
\ar[u]^{\widetilde{\partial}}\ar[r]&0,
}
\end{aligned}\end{align}
and
\begin{align}\begin{aligned}\label{e.423}
\xymatrix{
&&&\\
0\ar[r]&C^{\bullet}(W^{u}_{Z_{2}}/W^{u}_{Y}, F)\otimes \mathbb{C}^{-}\ar[u]^{\widetilde{\partial}}\ar[r]^{\quad\quad\psi_{2}^{-}}&C^{\bullet}(W^{u}_{\overline{Z_{2}}}, \overline{F})^{-}\ar[r]\ar[u]^{\widetilde{\overline{\partial}}}&0,
}
\end{aligned}\end{align}
are commutative. Moreover $\psi^{-}_{2}$ is isometric, but $\psi^{+}_{1}$ is not isometric with respect to the metrics induced by $h^{F}$ on these chain groups {\rm (cf. \cite[(1.34)]{BruMa12})}.
\end{lemma}

\begin{proof}
 By (\ref{e.636}), we know that the flow
 $\varphi_{t}$ ($t\geq 0$) of $\mathscr{Y}=-\nabla h$ can not pass through the frontier $Y$, i.e.,
 $$
\{\varphi_{t}(x)|t\geq 0\}\subset j_{1}(Z_{1}) (\text{resp.}\, j_{2}(Z_{1}))\quad \text{for}\, x\in j_{1}(Z_{1})\,(\text{resp.} \,j_{2}(Z_{1})).
$$
Let $\overline{{\bf{B}}_{1}}$ be the set of critical points of $\overline{h_{1}}$ and $\overline{B_{1}}$ be its fiber. Hence, for $x\in \overline{B_{1}}^{i}$, $x'\in \overline{B_{1}}^{i-1}$, if there exists an integral curve
 $\overline{\gamma}\in\Gamma_{\overline{Z_{1}}}(x, x')=W^{u}_{\overline{Z_{1}}}(x)\cap W^{s}_{\overline{Z_{1}}}(x')$, then
$x,\,x'$ and $\overline{\gamma}$ must stay in the same side of $Y$. By the above observation and (\ref{6.51}), for $x\in \overline{B}$ and $f^{*}\in \overline{F}^{*}_{x}$, $e\in o^{u}_{x}$ we get
\begin{align}\begin{aligned}\label{e.775}
\overline{\partial}(f^{*}\otimes e)=\partial(f^{*}\otimes e).
\end{aligned}\end{align}
Let $\mathfrak{a}^{*}\in C^{\bullet}(W^{u}_{\overline{Z_{1}}}, \overline{F})^{+}$, then for all $\mu\in C_{\bullet+1}(W^{u}_{Z_{1}}, F)$ by (\ref{6.50}) and (\ref{e.775}) we have
\begin{align}\begin{aligned}\label{e.776}
\sqrt{2}\langle\widetilde{\partial}\psi_{1}^{+}\mathfrak{a}^{*},\mu\rangle&=\langle j_{1}^{*}(\mathfrak{a}^{*}|_{Z_{1}}) +j_{2}^{*}(\mathfrak{a}^{*}|_{Z'_{1}}),\partial\mu\rangle=\langle \mathfrak{a}^{*}|_{Z_{1}},j_{1*}\partial\mu\rangle+\langle \mathfrak{a}^{*}|_{Z'_{1}},j_{2*}\partial\mu\rangle\\
&=\langle \mathfrak{a}^{*}|_{Z_{1}},\partial j_{1*}\mu\rangle+\langle \mathfrak{a}^{*}|_{Z'_{1}},\partial j_{2*}\mu\rangle=\langle \mathfrak{a}^{*},\overline{\partial} j_{1*}\mu\rangle+\langle \mathfrak{a}^{*},\overline{\partial} j_{2*}\mu\rangle\\
&=\langle \big(\widetilde{\overline{\partial}}\mathfrak{a}^{*}\big)|_{Z_{1}}, j_{1*}\mu\rangle+\langle \big(\widetilde{\overline{\partial}}\mathfrak{a}^{*}\big)|_{Z'_{1}}, j_{2*}\mu\rangle=\sqrt{2}\langle\psi_{1}^{+}\widetilde{\partial}\mathfrak{a}^{*},\mu\rangle,
\end{aligned}\end{align}
where $\overline{\partial}$ is the boundary operator of the Thom-Smale complex of $(\overline{M_{1}}, \overline{h_{1}})$ and $\widetilde{\overline{\partial}}$ denotes its dual. By (\ref{e.776}), we see that the diagram (\ref{e.422}) is commutative.

Now we show that the diagram (\ref{e.423}) is commutative. Let $\mathfrak{b}^{*}\in C^{\bullet}(W^{u}_{Z_{2}}/W^{u}_{Y}, F)$, for any $b\in C_{\bullet+1}(W^{u}_{\overline{Z_{2}}},\overline{F}^{*})$ we have
\begin{align}\begin{aligned}\label{e.812}
\sqrt{2}\langle\widetilde{\overline{\partial}}\psi_{2}^{-}\mathfrak{b}^{*},b\rangle&=\langle (j^{-1}_{1})^{*}\mathfrak{b}^{*} -(j^{-1}_{2})^{*}\mathfrak{b}^{*},\overline{\partial} b\rangle\\
&=\langle \mathfrak{b}^{*},j^{-1}_{1*}\big(\overline{\partial} b \cap Z_{2}\big)\rangle-\langle \mathfrak{b}^{*},j^{-1}_{2*}\big(\overline{\partial} b\cap Z'_{2}\big)\rangle
\end{aligned}\end{align}
and
\begin{align}\begin{aligned}\label{e.813}
\sqrt{2}\langle \psi_{2}^{-}\widetilde{\overline{\partial}}\mathfrak{b}^{*},b\rangle&=\langle (j^{-1}_{1})^{*}\widetilde{\overline{\partial}}\mathfrak{b}^{*} -(j^{-1}_{2})^{*}\widetilde{\overline{\partial}}\mathfrak{b}^{*},b\rangle\\
&=\langle \mathfrak{b}^{*},\overline{\partial}j^{-1}_{1*}\big( b \cap Z_{2}\big)\rangle-\langle \mathfrak{b}^{*},\overline{\partial}j^{-1}_{2*}\big( b\cap Z'_{2}\big)\rangle.
\end{aligned}\end{align}
Without lost of generality, we assume that $b\in Z'_{2}\backslash Y$ (for $b\in Y$ both (\ref{e.812}) and (\ref{e.813}) are null), hence by (\ref{e.775}) we see that
\begin{align}\begin{aligned}\label{e.814}
&j^{-1}_{2*}\big(\overline{\partial} b\cap Z'_{2}\big)=\overline{\partial}j^{-1}_{2*}\big( b\cap Z'_{2}\big),\\
&\overline{\partial}j^{-1}_{1*}\big( b \cap Z_{2}\big)=\emptyset,\quad \emptyset\neq j^{-1}_{1*}\big(\overline{\partial} b \cap Z_{2}\big)\in C_{\bullet}(W^{u}_{Y},F^{*}).
\end{aligned}\end{align}
By \cite[(1.32)]{BruMa12} and (\ref{e.814}), we have
\begin{align}\begin{aligned}\label{e.815}
\langle \mathfrak{b}^{*},j^{-1}_{1*}\big(\overline{\partial} b \cap Z_{2}\big)\rangle=\langle \mathfrak{b}^{*},\overline{\partial}j^{-1}_{1*}\big( b \cap Z_{2}\big)\rangle=0.
\end{aligned}\end{align}
By (\ref{e.812})--(\ref{e.815}), we prove that the diagram (\ref{e.423}) is commutative.

Let $\mathfrak{a}^{*}\in C^{\bullet}(W^{u}_{\overline{Z_{1}}}, \overline{F})$ be an element generated by critical points contained in $Y$, then we have
 \begin{align}\begin{aligned}\label{e.816}
\psi^{+}_{1}(\mathfrak{a}^{*})=\sqrt{2}\mathfrak{a}^{*}\otimes 1_{\mathbb{C}^{+}}.
\end{aligned}\end{align}
Similar to \cite[(2.10), (2.11)]{BruMa12}, we see that $\psi^{-}_{2}$ is isometric, but $\psi^{+}_{1}$ is not isometric with respect to the metric induced by $h^{F}$. The proof is completed.
\end{proof}

The double complex (\ref{e.422}) yields two spectral sequences: $({'E^{+}_{r}}, {{'d}}_{r})$ with the filtration (\ref{e.371}) and $( {''E^{+}_{r}}, {'{'d}}_{r})$ with the filtration (\ref{e.372}). The double complex (\ref{e.423}) yields also two spectral sequences: $({{'E}}^{-}_{r}, {{'d}}_{r})$ with the filtration (\ref{e.371}) and $( {{''E}}^{-}_{r}, {''d}_{r})$ with the filtration (\ref{e.372}).

By applying Theorem \ref{t.15} to the double complex (\ref{e.422}), we get the following lemma.

\begin{lemma}\label{l.40}The following identity holds in $Q^{S}/Q^{S, 0}$
\begin{align}\begin{aligned}\label{6.20}
T_{f}&(A^{C^{\bullet}(W_{\overline{Z_{1}}}^{u}, \overline{F})^{+}}, h^{C^{\bullet}(W_{\overline{Z_{1}}}^{u}, \overline{F})^{+}})
-\widetilde{f}\left(\nabla^{H^{\bullet}(\overline{Z_{1}}, \overline{F})^{+}}, h_{C^{\bullet}(W_{\overline{Z_{1}}}^{u}, \overline{F})^{+}}^{H^{\bullet}(\overline{Z_{1}}, \overline{F})^{+}},
 h_{L^{2}}^{H^{\bullet}(\overline{Z_{1}}, \overline{F})^{+}}\right)\\
&=T_{f}(A^{C^{\bullet}(W^{u}_{Z_{1}}, F)}, h^{C^{\bullet}(W^{u}_{Z_{1}}, F)})-\widetilde{f}\left(\nabla^{H^{\bullet}(Z_{1}, F)}, h_{C^{\bullet}(W^{u}_{Z_{1}}, F)}^{H^{\bullet}(Z_{1}, F)}, h_{L^{2}}^{H^{\bullet}(Z_{1}, F)}\right)\\
&\quad\quad\quad\quad\quad\quad-T_{f}(A_{1}^{{'E}^{+}_{1}}, h^{{'E}^{+}_{1}}_{L^{2}})+T_{f}(A^{{''E}^{+}_{0}}_{0}, h^{{''E}^{+}_{0}}).
\end{aligned}\end{align}
\end{lemma}
\begin{proof}
By (\ref{e.422}), we have
\begin{align}\begin{aligned}\label{6.21}
{''E}^{+}_{1}={{''E}^{+}_{\infty}}=0, \quad {{'E}^{+}_{2}}={{'E}^{+}_{\infty}}=0,
\end{aligned}\end{align}
hence by applying Theorem \ref{t.15} two times for the different filtrations, we get in $Q^{S}/Q^{S, 0}$
\begin{align}\begin{aligned}\label{6.18}
T_{f}(A^{{''E}^{+}_{0}}_{0}, h^{{''E}^{+}_{0}})=T_{f}(A^{{'E}^{+}_{0}}_{0}, h^{{'E}^{+}_{0}})+T_{f}(A^{{'E}^{+}_{1}}_{1}, h^{{'E}^{+}_{1}}),
\end{aligned}\end{align}
where for $0\leq q\leq m$
\begin{align}\begin{aligned}\label{6.19}
({'E}^{+}_{1})^{\cdot, q}&:\quad 0\longrightarrow H^{q}(\overline{Z_{1}}, \overline{F})^{+} \overset{\psi^{+}_{1}}\longrightarrow H^{q}(Z_{1}, F)\longrightarrow 0.
\end{aligned}\end{align}
In (\ref{6.19}) the isomorphism $\psi^{+}_{1}$ between the cohomology groups is induced by that of (\ref{e.422}).
From Theorem \ref{t.16}, (\ref{6.21}), (\ref{6.18}) and (\ref{6.19}), we get (\ref{6.20}).
\end{proof}

Apply Theorem \ref{t.15} to the double complex (\ref{e.423}), we get the following lemma.
\begin{lemma}\label{l.41}The following identity holds in $Q^{S}/Q^{S, 0}$
\begin{align}\begin{aligned}\label{6.59}
&T_{f}(A^{C^{\bullet}(W_{\overline{Z_{2}}}^{u}, \overline{F})^{-}}, h^{C^{\bullet}(W_{\overline{Z_{2}}}^{u}, \overline{F})^{-}})
-\widetilde{f}\big(\nabla^{H^{\bullet}(\overline{Z_{2}}, \overline{F})^{-}}, h_{C^{\bullet}(W_{\overline{Z_{2}}}^{u}, \overline{F})^{-}}^{H^{\bullet}(\overline{Z_{2}}, \overline{F})^{-}}, h_{L^{2}}^{H^{\bullet}(\overline{Z_{2}}, \overline{F})^{-}}\big)\\
=&T_{f}(A^{C^{\bullet}(W^{u}_{Z_{2}}/W_{Y}^{u}, F)}, h^{C^{\bullet}(W^{u}_{Z_{2}}/W_{Y}^{u}, F)})\\
&-\widetilde{f}\big(\nabla^{H^{\bullet}(Z_{2}, Y, F)}, h_{C^{\bullet}(W^{u}_{Z_{2}}/W_{Y}^{u}, F)}^{H^{\bullet}(Z_{2}, Y, F)}, h_{L^{2}}^{H^{\bullet}(Z_{2}, Y, F)}\big)
+T_{f}(A^{{'E}^{-}_{1}}, h^{{'E}^{-}_{1}}_{L^{2}}).
\end{aligned}\end{align}
\end{lemma}
\begin{proof}
 By (\ref{e.423}), we have
\begin{align}\begin{aligned}\label{6.60}
{''E}^{-}_{1}={''E}^{-}_{\infty}=0, \quad {'E}^{-}_{2}={'E}^{-}_{\infty}=0.
\end{aligned}\end{align}
By Theorem \ref{t.15}, we get in $Q^{S}/Q^{S, 0}$
\begin{align}\begin{aligned}\label{6.61}
T_{f}(A^{{''E}^{-}_{0}}_{0}, h^{{''E}^{-}_{0}})=T_{f}(A^{{'E}^{-}_{0}}_{0}, h^{{'E}^{-}_{0}})+T_{f}(A^{{'E}^{-}_{1}}_{1}, h^{{'E}^{-}_{1}}),
\end{aligned}\end{align}
where for $0\leq q\leq m$
\begin{align}\begin{aligned}\label{6.62}
({'E}^{-}_{1})^{\cdot, q}&:\quad 0\longrightarrow H^{q}(Z_{2}, Y, F)\overset{\psi_{2}^{-}}\longrightarrow H^{q}(\overline{Z_{2}}, \overline{F})^{-}\longrightarrow 0.
\end{aligned}\end{align}
In (\ref{6.62}), the isometric $\psi_{2}^{-}$ between the cohomology groups is induced by that of (\ref{e.423}) in Lemma \ref{l.38}. Since $\psi_{2}^{-}$ is an isometric and $({''E}^{-}_{0})^{\cdot, q}$, $({'E}^{-}_{1})^{\cdot, q}$ are exact, they split.
From \cite[Thm. A1.1]{BL}, it follows that
\begin{align}\begin{aligned}\label{6.63}
T_{f}(A^{{''E}^{-}_{0}}_{0}, h^{{''E}^{-}_{0}})=0, \quad T_{f}(A^{{'E}^{-}_{1}}_{1}, h^{{'E}^{-}_{1}})=0.
\end{aligned}\end{align}
From Theorem \ref{t.16}, (\ref{6.60}), (\ref{6.61}), (\ref{6.62}) and (\ref{6.63}), we get (\ref{6.59}).
\end{proof}

Next we try to compute the fourth term $T_{f}(A^{{''E}^{+}_{0}}_{0}, h^{{''E}^{+}_{0}})$ at the right side of (\ref{6.20}).

\begin{lemma}\label{l6.1} For $({''E}^{+}_{0},{''d}_{0})$, we have in $Q^{S}/Q^{S, 0}$
\begin{align}\begin{aligned}\label{6.22}
T_{f}(A^{{''E}^{+}_{0}}_{0}, h^{{''E}^{+}_{0}})=-\frac{\log 2}{2}\chi(Y)\rk(F).
\end{aligned}\end{align}
\end{lemma}
\begin{proof} By (\ref{e.422}) and Lemma \ref{l.40}, we have the exact sequence
\begin{align}\begin{aligned}\label{6.25}
({''E}^{+}_{0})^{\cdot, q}:\quad 0\longrightarrow C^{q}(W_{\overline{Z_{1}}}^{u}, \overline{F})^{+}\overset{\psi^{+}_{1}}\longrightarrow C^{q}(W^{u}_{Z_{1}}, F) \longrightarrow 0.
\end{aligned}\end{align}
For short, we use $V^{q}_{1}\subset C^{q}(W_{\overline{Z_{1}}}^{u}, \overline{F})^{+}$ to denote the subspace of $C^{q}(W_{\overline{Z_{1}}}^{u}, \overline{F})^{+}$ generated by critical points in $\overline{Z_{1}}\backslash Y$ and $V^{q}_{2}\subset C^{q}(W_{\overline{Z_{1}}}^{u}, \overline{F})^{+}$ to denote the subspace generated by critical points in $Y$. Since all the Hermitian metrics appearing here are induced by $h^{F}$, we have the orthogonal decomposition $C^{q}(W_{\overline{Z_{1}}}^{u}, \overline{F})^{+}=V^{q}_{1}\oplus V^{q}_{2}$.
We define a map $\alpha: V^{q}_{1}\oplus V^{q}_{2}\rightarrow C^{q}(W^{u}_{Z_{1}}, F)$ by
\begin{align}\begin{aligned}\label{e.795}
\alpha(x)=\left\{
\begin{array}{cc}
  \psi^{+}_{1}(x) & \text{for}\quad x\in V^{q}_{1}, \\
  x &  \text{for}\quad x\in V^{q}_{2}.
\end{array}
\right.
 \end{aligned}\end{align}
By \cite[(2.10), (2.11)]{BruMa12} and (\ref{6.50}), we see that $\alpha$ is an isometric isomorphism. If we use $\alpha$ to identify $V^{q}_{1}\oplus V^{q}_{2}$ with $C^{q}(W_{Z_{1}}^{u}, F)$, then the exact sequence (\ref{6.25}) is equivalent to
\begin{align}\begin{aligned}\label{e.796}
\mathscr{V}_{q}: \quad 0\rightarrow V^{q}_{1}\oplus V^{q}_{2}\overset{\widetilde{\psi^{+}_{1}}}\rightarrow V^{q}_{1}\oplus V^{q}_{2}\rightarrow 0,\quad \text{with}\quad \widetilde{\psi^{+}_{1}}=\left(
                                            \begin{array}{cc}
                                              \Id|_{V^{q}_{1}} & 0 \\
                                              0 & \sqrt{2}\Id|_{V^{q}_{2}} \\
                                            \end{array}
                                          \right).
 \end{aligned}\end{align}

By Lemma \ref{l.49}, we get
in $Q^{S}/Q^{S, 0}$
\begin{align}\begin{aligned}\label{6.24}
T_{f}&(A^{{''E}^{+}_{0}}_{0}, h^{{''E}^{+}_{0}})=\sum_{q=0}^{m}(-1)^{q}T_{f}(A^{({''E}^{+}_{0})^{\cdot, q}}_{0}, h^{({''E}^{+}_{0})^{\cdot, q}})
=\sum_{q=0}^{m}(-1)^{q}T_{f}(A^{\mathscr{V}_{q}},h^{\mathscr{V}_{q}})\\
&=-\frac{\log 2}{2}\sum_{q=0}^{m-1}(-1)^{q}\rk \Big(C^{q}(W^{u}_{Y}, F)\Big)=-\frac{\log 2}{2}\chi(Y)\rk(F),
\end{aligned}\end{align}
 by using the fact that $\sum_{q=0}^{m-1}(-1)^{q}\cdot
\rk\Big(C^{q}(W^{u}_{Y}, F)\Big)=\chi(Y)\rk(F)$ (cf. \cite[(2.17)]{BruMa12}. Now the proof of Lemma \ref{6.22} is completed.
\end{proof}

\subsection{Computations of $T_{f}(A^{{'E}^{+}_{1}}, h^{{'E}^{+}_{1}}_{L^{2}})$ in (\ref{6.20}) and $T_{f}(A^{{'E}^{-}_{1}}, h^{{'E}^{-}_{1}}_{L^{2}})$ in (\ref{6.59})}\label{ss2.7}
For $i=1, 2$, let $\mathscr{H}(Z_{i}, F)$ denote the space of harmonic forms on $Z_{i}$ satisfying the absolute boundary conditions,
 $\mathscr{H}(Z_{i}, Y, F)$ for that with the relative boundary conditions and $\mathscr{H}(\overline{Z_{i}}, \overline{F}_{i})$
for the space of harmonic forms on $\overline{Z_{i}}$.
As \cite[Prop. 2.1]{BruMa12}, we have a natural isometry of $\mathbb{Z}_{2}$-vector spaces with respect to $L^{2}-$metrics,
\begin{align}\begin{aligned}\label{6.70}
\widetilde{\phi}_{i}&:\quad \mathscr{H}(\overline{Z_{i}}, \overline{F})\longrightarrow \mathscr{H}(Z_{i}, F)\otimes 1_{\mathbb{C}^{+}}\oplus \mathscr{H}(Z_{i}, Y, F)\otimes 1_{\mathbb{C}^{-}}, \\
&\widetilde{\phi}_{i}(\sigma)=\frac{\sqrt{2}}{2}\cdot(\sigma+\phi^{*}_{i}\sigma)|_{Z_{i}} +\frac{\sqrt{2}}{2}\cdot(\sigma-\phi^{*}_{i}\sigma)|_{Z_{i}}.
\end{aligned}\end{align}
Let $P_{\infty}$ be the de Rham map (cf. \cite[(1.15)]{BruMa12}). For $\pi:M_{1}\rightarrow S$, the isometry $\widetilde{\phi}_{1}$ induces the isometry:
\begin{align}\begin{aligned}\label{6.71}
P_{\infty}\circ\widetilde{\phi}_{1}\circ (P_{\infty})^{-1}:\quad H^{\bullet}(\overline{Z_{1}}, \overline{F})^{+}\longrightarrow H^{\bullet}(Z_{1}, F)\otimes 1_{\mathbb{C}^{+}}.\end{aligned}\end{align}
For $\pi:M_{2}\rightarrow S$, the isometry $\widetilde{\phi}_{2}$ induces the isometry:
\begin{align}\begin{aligned}\label{6.72}
P_{\infty}\circ\widetilde{\phi}_{2}\circ (P_{\infty})^{-1}:\quad H^{\bullet}(\overline{Z_{2}}, \overline{F})^{-}\longrightarrow H^{\bullet}(Z_{2}, Y, F)\otimes 1_{\mathbb{C}^{-}}.
\end{aligned}\end{align}

For $\sigma_{1}\in H^{\bullet}(\overline{Z_{1}},\overline{F}),\, \sigma_{2}\in H^{\bullet}(Z_{2},Y,F)$, we have

\begin{align}\begin{aligned}\label{e.817}
&P_{\infty}\circ\widetilde{\phi}_{1}\circ (P_{\infty})^{-1}(\sigma_{1})|_{C_{\bullet}(W^{u}_{Z_{1}},F^{*})}=
\frac{\sqrt{2}}{2}\big(\sigma_{1}+\phi^{*}\sigma_{1}\big)|_{C_{\bullet}(W^{u}_{Z_{1}},F^{*})}=\psi^{+}_{1}\sigma_{1}|_{C_{\bullet}(W^{u}_{Z_{1}},F^{*})},\\
&P_{\infty}\circ\widetilde{\phi}_{2}\circ (P_{\infty})^{-1}\circ \psi^{-}_{2}(\sigma_{2}\otimes 1_{\mathbb{C}^{-}})|_{C_{\bullet}(W^{u}_{Z_{2}}/W^{u}_{Y},\overline{F}^{*})}=\sigma_{2}|_{C_{\bullet}(W^{u}_{Z_{2}}/W^{u}_{Y},\overline{F}^{*})}.
\end{aligned}\end{align}

\begin{lemma}\label{l.55} We have the following identities in $Q^{S}/Q^{S, 0}$
\begin{subequations}
\begin{align}
T_{f}(A^{{'E}^{+}_{1}}, h^{{'E}^{+}_{1}}_{L^{2}})=0,\label{6.85}\\
T_{f}(A^{{'E}^{-}_{1}}, h^{{'E}^{-}_{1}}_{L^{2}})=0.\label{6.91}
\end{align}
\end{subequations}
\end{lemma}

\begin{proof}
By the first equation of (\ref{e.817}), the isomorphisms $P_{\infty}\circ\widetilde{\phi}_{1}\circ (P_{\infty})^{-1}$ and $\psi^{+}_{1}$
are the same maps from $H^{\bullet}(\overline{Z_{1}}, \overline{F})^{+}$ to $H^{\bullet}(Z_{1}, F)$, hence $\psi_{1}^{+}$ is isometric with
respect to the $L^{2}-$metrics. This means that the exact sequences $\big(({'E^{+}_{1}})^{\cdot,q},h^{({'E^{+}_{1}})^{\cdot,q}}_{L^{2}}\big),\,0\leq q \leq m$ (see (\ref{6.19})) split. Then (\ref{6.85}) follows from \cite[Thm. A1.1]{BL} or Lemma \ref{l.49}.

Next we try to compute the term $T_{f}(A^{{'E}^{-}_{1}}, h^{{'E}^{-}_{1}}_{L^{2}})$ appearing in (\ref{6.59}) of Lemma \ref{l.41}. In this case
${'E}^{-}_{1}$ represents the exact sequence induced by the horizontal lines of diagram (\ref{e.423}):
\begin{align}\begin{aligned}\label{6.86}
({'E}^{-}_{1})^{\cdot,q}:\quad 0\longrightarrow H^{q}(Z_{2}, Y, F)\overset{\psi_{2}^{-}}\longrightarrow H^{q}(\overline{Z_{2}}, \overline{F})^{-}\longrightarrow 0.\end{aligned}\end{align}
We use a new notation to represent the exact sequence (\ref{6.72}), which splits with respect to $L^{2}-$metric, by
\begin{align}\begin{aligned}\label{6.87}
\mathscr{P}^{q}_{2}: 0\longrightarrow H^{q}(\overline{Z_{2}}, \overline{F})^{-}\overset{P_{\infty}\circ\widetilde{\phi}_{2}\circ (P_{\infty})^{-1}}{\longrightarrow} H^{q}(Z_{2}, Y, F)\longrightarrow 0.
\end{aligned}\end{align}
 Then the composition of (\ref{6.86}) and (\ref{6.87}) yields
\begin{align}\begin{aligned}\label{6.88}
\mathscr{O}^{q}_{2}=({'E}^{-}_{1})^{\cdot,q}\circ\mathscr{P}^{q}_{2}:0\longrightarrow H^{q}(\overline{Z_{2}}, \overline{F})^{-}\overset{w}\longrightarrow H^{q}(\overline{Z_{2}}, \overline{F})^{-}\longrightarrow 0,
\end{aligned}\end{align}
where $w=\psi_{2}^{-}\circ P_{\infty}\circ\widetilde{\phi}_{2}\circ P_{\infty}^{-1}$ is an identity map by the second equation of (\ref{e.817}).
As the sequences (\ref{6.87}) and (\ref{6.88}) split, by \cite[Thm. A1.1]{BL} or Lemma \ref{l.49} we get
\begin{align}\begin{aligned}\label{6.89}
T_{f}(A^{\mathscr{P}^{q}_{2}}, h^{\mathscr{P}^{q}_{2}}_{L^{2}})=0, \quad T_{f}(A^{\mathscr{O}^{q}_{2}}, h^{\mathscr{O}^{q}_{2}}_{L^{2}})=0.\end{aligned}\end{align}
By Lemma \ref{l.36} we find that
\begin{align}\begin{aligned}\label{6.90}
T_{f}(A^{\mathscr{O}^{q}_{2}}, h^{\mathscr{O}^{q}_{2}}_{L^{2}})=T_{f}(A^{\mathscr{P}^{q}_{2}}, h^{\mathscr{P}^{q}_{2}}_{L^{2}})+T_{f}(A^{({'E}^{-}_{1})^{\cdot,q}}, h^{({'E}^{-}_{1})^{\cdot,q}}_{L^{2}}),
\end{aligned}\end{align}
then from (\ref{6.86}), (\ref{6.89}) and (\ref{6.90}) we get (\ref{6.91}).
\end{proof}

\subsection{Comparison of analytic torsion forms and combinatorial torsion forms}\label{ss2.8}
By \cite[Thm. 0.1]{BGo}, the following identity holds in $Q^{S}/Q^{S, 0}$, for $i=1, 2$,
\begin{align}\begin{aligned}\label{6.7}
\mathscr{T}_{g}&(T^{H}\overline{M}_{i}, g^{T\overline{Z_{i}}}, h^{\overline{F}_{i}})-T_{f, g}\big(A^{C^{\bullet}(W^{u}_{\overline{Z_{i}}}, \overline{F}_{i})}, h^{C^{\bullet}(W^{u}_{\overline{Z_{i}}}, \overline{F}_{i})}\big)\\
&+\widetilde{f}_{g}\big(\nabla^{H^{\bullet}(\overline{Z_{i}}, \overline{F}_{i})}, h_{C^{\bullet}(W^{u}_{\overline{Z_{i}}}, \overline{F}_{i})}^{H^{\bullet}(\overline{Z_{i}}, \overline{F}_{i})}, h_{L^{2}}^{H^{\bullet}(\overline{Z_{i}}, \overline{F}_{i})}\big)\\
=&-\int_{\overline{Z}_{i, g}}f_{g}(\nabla^{\overline{F}}, h^{\overline{F}})(\nabla h)^{*}\psi(T\overline{Z}_{i, g}, \nabla^{T\overline{Z}_{i, g}})\\
&+\sum_{x\in \overline{B}_{i, g}}(-1)^{\text{ind}(x)}\tr ^{\overline{F}_{x}\otimes \overline{o}^{u}_{x}}[g]\cdot{^{0}I_{g}(T_{x}\overline{Z_{i}}|_{\overline{\mathbf{B}}_{i, g}})},
\end{aligned}\end{align}
where $g\in \mathbb{Z}_{2}=\{1, \phi_{i}\}$ and ${^{0}I_{g}}$ is an additive class (cf. \cite[(0.20)]{BGo}).
For the third term of the left side of the identity (\ref{6.7}), we have the following identity from the decomposition of
$H^{\bullet}(\overline{Z_{i}}, \overline{F}_{i})$ into the eigen-subbundles with respect to the action of the involution $\phi_{i}$, i.e.,
\begin{align}\begin{aligned}\label{6.9}
\widetilde{f}_{g}\big(\nabla^{H^{\bullet}(\overline{Z_{i}}, \overline{F}_{i})},& h_{C^{\bullet}(W^{u}_{\overline{Z_{i}}}, \overline{F}_{i})}^{H^{\bullet}(\overline{Z_{i}}, \overline{F}_{i})}, h_{L^{2}}^{H^{\bullet}(\overline{Z_{i}}, \overline{F}_{i})}\big)\\
=\widetilde{f}\big(\nabla^{H^{\bullet}(\overline{Z_{i}}, \overline{F}_{i})^{+}},& h_{C^{\bullet}(W^{u}_{\overline{Z_{i}}}, \overline{F}_{i})^{+}}^{H^{\bullet}(\overline{Z_{i}}, \overline{F}_{i})^{+}}, h_{L^{2}}^{H^{\bullet}(\overline{Z_{i}}, \overline{F}_{i})^{+}}\big)\\
&+\chi(g)\widetilde{f}\big(\nabla^{H^{\bullet}(\overline{Z_{i}}, \overline{F}_{i})^{-}}, h_{C^{\bullet}(W^{u}_{\overline{Z_{i}}}, \overline{F}_{i})^{-}}^{H^{\bullet}(\overline{Z_{i}}, \overline{F}_{i})^{-}}, h_{L^{2}}^{H^{\bullet}(\overline{Z_{i}}, \overline{F}_{i})^{-}}\big).
\end{aligned}\end{align}

Recall that $N$ denotes the normal bundle of $X\subset M$. We have
$TZ|_{X}=N\oplus TY$, so we get $TZ|_{\mathbf{B}\cap X}=N|_{\mathbf{B}\cap X}\oplus TY|_{\mathbf{B}\cap X}$ regarded as vector bundles over $S$ as direct image of $\pi: \mathbf{B}\cap X\rightarrow S$. Since ${^{0}I_{g}}$ is an additive class (cf. \cite[(0.20)]{BGo}), we get
\begin{align}\begin{aligned}\label{e.756}
{^{0}I_{g}}(TZ|_{\mathbf{B}\cap X})={^{0}I_{g}}(N|_{\mathbf{B}\cap X})+{^{0}I_{g}}(TY|_{\mathbf{B}\cap X}).
\end{aligned}\end{align}
(Since $N|_{\mathbf{B}\cap X}$ is a trivial line bundle over $S$, the cohomology class of ${^{0}I_{g}}(N|_{\mathbf{B}\cap X})$ vanishes.) From (\ref{6.3}), (\ref{6.6}), (\ref{6.7}) and (\ref{6.9}), we get for $Z\mathrm{}=Z_{1}\cup_{Y} Z_{2}$
\begin{align}\begin{aligned}\label{6.10}
\mathscr{T}&(T^{H}M, g^{TZ}, h^{F})-T_{f}(A^{C^{\bullet}(W^{u}_{Z}, F)}, h^{C^{\bullet}(W^{u}_{Z}, F)})\\
&+\widetilde{f}\big(\nabla^{H^{\bullet}(Z, F)}, h_{C^{\bullet}(W^{u}, F)}^{H^{\bullet}(Z, F)}, h_{L^{2}}^{H^{\bullet}(Z, F)}\big)\\
=&-\int_{Z}f(\nabla^{F}, h^{F})(\nabla h)^{*}\psi(TZ, \nabla^{TZ})+\sum_{x\in B}(-1)^{\text{ind}(x)}\rk(F)\cdot {^{0}I(T_{x}Z|_{\mathbf{B}})},
\end{aligned}\end{align}
\begin{align}\begin{aligned}\label{6.11}
&\mathscr{T}_{{\rm abs}}(T^{H}M_{1}, g^{TZ_{1}}, h^{F})
-T_{f}(A^{C^{\bullet}(W_{\overline{Z_{1}}}^{u}, \overline{F})^{+}}, h^{C^{\bullet}(W_{\overline{Z_{1}}}^{u}, \overline{F})^{+}})\\
&+\widetilde{f}\big(\nabla^{H^{\bullet}(W_{\overline{Z_{1}}}^{u}, \overline{F})^{+}}, h_{C^{\bullet}(W_{\overline{Z_{1}}}^{u}, \overline{F})^{+}}^{H^{\bullet}(\overline{Z_{1}}, \overline{F})^{+}}, h_{L^{2}}^{H^{\bullet}(\overline{Z_{1}}, \overline{F})^{+}}\big)\\
=&-\int_{Z_{1}}f(\nabla^{F}, h^{F})(\nabla f)^{*}\psi(TZ_{1}, \nabla^{TZ_{1}})+\frac{1}{2}\rk(F)\sum_{x\in \overline{B}_{1}}(-1)^{\text{ind}(x)} \cdot{^{0}I(T_{x}\overline{Z_{1}}|_{\overline{\mathbf{B}}_{1}})}\\
&-\frac{1}{2}\int_{Y}f(\nabla^{\overline{F}}, h^{\overline{F}})(\nabla f)^{*}\psi(TY, \nabla^{TY})+\frac{1}{2}\rk(F)\sum_{x\in B_{Y}}(-1)^{\text{ind}(x)} \cdot{^{0}I(T_{x}Y|_{\mathbf{B}\cap X})}\\
&+\frac{1}{2}\rk(F)\sum_{x\in B_{Y}}(-1)^{\text{ind}(x)}\cdot {^{0}I_{\phi}(N_{x}|_{\mathbf{B}\cap X})},
\end{aligned}\end{align}
and
\begin{align}\begin{aligned}\label{6.12}
&\mathscr{T}_{{\rm rel}}(T^{H}M_{2}, g^{TZ_{2}}, h^{F})-
T_{f}(A^{C^{\bullet}(W_{\overline{Z_{2}}}^{u}, \overline{F})^{-}}, h^{C^{\bullet}(W_{\overline{Z_{2}}}^{u}, \overline{F})^{-}})\\
&\qquad\qquad+\widetilde{f}\big(\nabla^{H^{\bullet}(\overline{Z_{2}}, \overline{F})^{-}}, h_{C^{\bullet}(W_{\overline{Z_{2}}}^{u}, \overline{F})^{-}}^{H^{\bullet}(\overline{Z_{2}}, \overline{F})^{-}}, h_{L^{2}}^{H^{\bullet}(\overline{Z_{2}}, \overline{F})^{-}}\big)\\
=&-\int_{Z_{2}}f(\nabla^{F}, h^{F})(\nabla f)^{*}\psi(TZ_{2}, \nabla^{TZ_{2}})+\frac{1}{2}\rk(F)\sum_{x\in \overline{B}_{2}}(-1)^{\text{ind}(x)} \cdot{^{0}I(T_{x}\overline{Z_{2}}|_{\overline{\mathbf{B}}_{2}})}\\
&+\frac{1}{2}\int_{Y}f(\nabla^{\overline{F}}, h^{\overline{F}})(\nabla f)^{*}\psi(TY, \nabla^{TY})-\frac{1}{2}\rk(F)\sum_{x\in B_{Y}}(-1)^{\text{ind}(x)} \cdot{^{0}I(T_{x}Y|_{\mathbf{B}\cap X})}\\
&-\frac{1}{2}\rk(F)\sum_{x\in B_{Y}}(-1)^{\text{ind}(x)}\cdot {^{0}I_{\phi}(N_{x}|_{\mathbf{B}\cap X})}.
\end{aligned}\end{align}
From (\ref{e.756}), (\ref{6.10}), (\ref{6.11}) and (\ref{6.12}), we get
\begin{align}\begin{aligned}\label{6.13}
&\mathscr{T}(T^{H}M, g^{TZ}, h^{F})-\mathscr{T}_{{\rm abs}}(T^{H}M_{1}, g^{TZ_{1}}, h^{F})-\mathscr{T}_{{\rm rel}}(T^{H}M_{2}, g^{TZ_{2}}, h^{F})\\
=&T_{f}(A^{C^{\bullet}(W^{u}_{Z}, F)}, h^{C^{\bullet}(W^{u}_{Z}, F)})
-T_{f}(A^{C^{\bullet}(W_{\overline{Z_{1}}}^{u}, \overline{F})^{+}}, h^{C^{\bullet}(W_{\overline{Z_{1}}}^{u}, \overline{F})^{+}})\\
&-T_{f}(A^{C^{\bullet}(W_{\overline{Z_{2}}}^{u}, \overline{F})^{-}}, h^{C^{\bullet}(W_{\overline{Z_{2}}}^{u}, \overline{F})^{-}})
-\widetilde{f}\big(\nabla^{H^{\bullet}(Z, F)}, h_{C^{\bullet}(W^{u}_{Z}, F)}^{H^{\bullet}(Z, F)}, h_{L^{2}}^{H^{\bullet}(Z, F)}\big)\\
&+\widetilde{f}\big(\nabla^{H^{\bullet}(\overline{Z_{1}}, \overline{F})^{+}}, h_{C^{\bullet}(W_{\overline{Z_{1}}}^{u}, \overline{F})^{+}}^{H^{\bullet}(\overline{Z_{1}}, \overline{F})^{+}}, h_{L^{2}}^{H^{\bullet}(\overline{Z_{1}}, \overline{F})^{+}}\big)\\
&+\widetilde{f}\big(\nabla^{H^{\bullet}(\overline{Z_{2}}, \overline{F})^{-}}, h_{C^{\bullet}(W_{\overline{Z_{2}}}^{u}, \overline{F})^{-}}^{H^{\bullet}(\overline{Z_{2}}, \overline{F})^{-}}, h_{L^{2}}^{H^{\bullet}(\overline{Z_{2}}, \overline{F})^{-}}\big).
\end{aligned}\end{align}
Recall that ${'E}_{0}$ is defined in (\ref{e.375}) and ${'E}_{1}$ in (\ref{e.376}), hence we have
 \begin{align}\begin{aligned}\label{6.94}
 T_{f}\big(A^{{'E}_{0}}_{0}, h^{{'E}_{0}}\big)=&T_{f}(A^{C^{\bullet}(W^{u}_{Z_{2}}/W_{Y}^{u}, F)}, h^{C^{\bullet}(W^{u}_{Z_{2}}/W_{Y}^{u}, F)})\\
&-T_{f}(A^{C^{\bullet}(W^{u}_{Z}, F)}, h^{C^{\bullet}(W^{u}_{Z}, F)})+T_{f}(A^{C^{\bullet}(W^{u}_{Z_{1}}, F)}, h^{C^{\bullet}(W^{u}_{Z_{1}}, F)}),
 \end{aligned}\end{align}
and
 \begin{align}\begin{aligned}\label{6.95}
 \widetilde{f}\big(\nabla^{{'E}_{1}}, h^{{'E}_{1}}_{L^{2}}, h^{{'E}_{1}}\big)=&
 \widetilde{f}\big(\nabla^{H^{\bullet}(Z, F)}, h_{C^{\bullet}(W^{u}_{Z}, F)}^{H^{\bullet}(Z, F)}, h_{L^{2}}^{H^{\bullet}(Z, F)}\big)\\
&-\widetilde{f}\big(\nabla^{H^{\bullet}(Z_{1}, F)}, h_{C^{\bullet}(W^{u}_{Z_{1}}, F)}^{H^{\bullet}(Z_{1}, F)}, h_{L^{2}}^{H^{\bullet}(Z_{1}, F)}\big)\\
&-\widetilde{f}\big(\nabla^{H^{\bullet}(Z_{2}, Y, F)}, h_{C^{\bullet}(W^{u}_{Z_{2}}/W_{Y}^{u}, F)}^{H^{\bullet}(Z_{2}, Y, F)}, h_{L^{2}}^{H^{\bullet}(Z_{2}, Y, F)}\big).
\end{aligned}\end{align}
By Lemmas \ref{l.40}, \ref{l.41}, \ref{l6.1}, \ref{l.55}, (\ref{6.13}), (\ref{6.94}) and (\ref{6.95}), we get the equation (\ref{5.11}).
The proof of Theorem \ref{t.17} is completed.

\bibliographystyle{plain}

\end{document}

%% file: fibration.pstex_t
\begin{picture}(0,0)%
\includegraphics{fibration.pstex}%
\end{picture}%
\setlength{\unitlength}{3947sp}%
\begingroup\makeatletter\ifx\SetFigFont\undefined%
\gdef\SetFigFont#1#2#3#4#5{%
  \reset@font\fontsize{#1}{#2pt}%
  \fontfamily{#3}\fontseries{#4}\fontshape{#5}%
  \selectfont}%
\fi\endgroup%
\begin{picture}(4380,2919)(4221,-6073)
\put(7501,-3736){\makebox(0,0)[lb]{\smash{{\SetFigFont{12}{14.4}{\rmdefault}{\mddefault}{\updefault}{\color[rgb]{0,0,0}$M_{2}$}%
}}}}
\put(6226,-4936){\makebox(0,0)[lb]{\smash{{\SetFigFont{12}{14.4}{\rmdefault}{\mddefault}{\updefault}{\color[rgb]{0,0,0}$\pi$}%
}}}}
\put(6001,-5911){\makebox(0,0)[lb]{\smash{{\SetFigFont{12}{14.4}{\rmdefault}{\mddefault}{\updefault}{\color[rgb]{0,0,0}$S$}%
}}}}
\put(5551,-4111){\makebox(0,0)[lb]{\smash{{\SetFigFont{8}{9.6}{\rmdefault}{\mddefault}{\updefault}{\color[rgb]{0,0,0}$-\varepsilon$}%
}}}}
\put(6451,-4111){\makebox(0,0)[lb]{\smash{{\SetFigFont{8}{9.6}{\rmdefault}{\mddefault}{\updefault}{\color[rgb]{0,0,0}$\varepsilon$}%
}}}}
\put(4651,-3736){\makebox(0,0)[lb]{\smash{{\SetFigFont{12}{14.4}{\rmdefault}{\mddefault}{\updefault}{\color[rgb]{0,0,0}$M_{1}$}%
}}}}
\put(6076,-3886){\makebox(0,0)[lb]{\smash{{\SetFigFont{12}{14.4}{\rmdefault}{\mddefault}{\updefault}{\color[rgb]{0,0,0}$X$}%
}}}}
\put(5476,-3286){\makebox(0,0)[lb]{\smash{{\SetFigFont{12}{14.4}{\rmdefault}{\mddefault}{\updefault}{\color[rgb]{0,0,0}$M=M_{1}\cup_{X}M_{2}$}%
}}}}
\end{picture}%

%% file: fibration_left.pstex_t
\begin{picture}(0,0)%
\includegraphics{fibration_left.pstex}%
\end{picture}%
\setlength{\unitlength}{3947sp}%
\begingroup\makeatletter\ifx\SetFigFont\undefined%
\gdef\SetFigFont#1#2#3#4#5{%
  \reset@font\fontsize{#1}{#2pt}%
  \fontfamily{#3}\fontseries{#4}\fontshape{#5}%
  \selectfont}%
\fi\endgroup%
\begin{picture}(4380,2919)(4221,-6073)
\put(6226,-4936){\makebox(0,0)[lb]{\smash{{\SetFigFont{12}{14.4}{\rmdefault}{\mddefault}{\updefault}{\color[rgb]{0,0,0}$\pi$}%
}}}}
\put(6001,-5911){\makebox(0,0)[lb]{\smash{{\SetFigFont{12}{14.4}{\rmdefault}{\mddefault}{\updefault}{\color[rgb]{0,0,0}$S$}%
}}}}
\put(5551,-4111){\makebox(0,0)[lb]{\smash{{\SetFigFont{8}{9.6}{\rmdefault}{\mddefault}{\updefault}{\color[rgb]{0,0,0}$-\varepsilon$}%
}}}}
\put(6451,-4111){\makebox(0,0)[lb]{\smash{{\SetFigFont{8}{9.6}{\rmdefault}{\mddefault}{\updefault}{\color[rgb]{0,0,0}$0$}%
}}}}
\put(4651,-3736){\makebox(0,0)[lb]{\smash{{\SetFigFont{12}{14.4}{\rmdefault}{\mddefault}{\updefault}{\color[rgb]{0,0,0}$M$}%
}}}}
\put(6498,-3730){\makebox(0,0)[lb]{\smash{{\SetFigFont{12}{14.4}{\rmdefault}{\mddefault}{\updefault}{\color[rgb]{0,0,0}$X$}%
}}}}
\end{picture}%